%% file: openlt_rev.tex
\documentclass[12pt,a4paper]{article}
\usepackage[dvips]{graphicx}
\usepackage{graphics,color}
\usepackage{enumerate}
\usepackage{amssymb,amsmath,amsthm,amscd,latexsym}

\newtheorem{theo}{Theorem}[section]
\newtheorem{addendum}[theo]{Addendum}
\newtheorem{corol}[theo]{Corollary}
\newtheorem{prop}[theo]{Proposition}

\newtheorem{lem}[theo]{Lemma}

\theoremstyle{definition}
\newtheorem{defi}[theo]{Definition}
\newtheorem{rem}[theo]{Remark}

\newtheorem{claim}{Claim}

\def\Rr{\mathbf{R}}
\let\RR\Rr

\def\Zz{\mathbf{Z}}

\let\NN\Nn

\def\calc{\mathcal{C}}

\let\bord\partial
\let\bydef\emph
\let\epsi\varepsilon

\def\reg{\mathrm{reg}}
\def\Ric{\mathrm{Ric}}
\def\Rm{\mathrm{Rm}}
\def\Riem{\mathrm{Riem}}
\def\vol{\mathrm{vol}}
\def\hyp{\mathrm{hyp}}
\def\cusp{\mathrm{cusp}}
\def\area{\mathrm{area}}

\def\Rmin{R_\mathrm{min}}
\def\Rmax{R_\mathrm{max}}

\def\Int{\mathrm{Int}}

\def\Tmax{T_\mathrm{max}}

\def\Ric{\mathop{\rm Ric}\nolimits}
\def\Rm{\mathop{\rm Rm}\nolimits}
\def\inj{\mathop{\rm inj}\nolimits}
\def\im{\mathop{\rm im}\nolimits}
\def\Int{\mathop{\rm int}\nolimits}
\def\tr{\mathop{\rm tr}\nolimits}

\def\const{\mathrm{const}}
\def\length{\mathrm{length}}
\def\area{\mathrm{area}}
\def\reg{\mathrm{reg}}

\def\eucl{\mathrm{eucl}}

\newcommand{\calm}{{\mathcal{M}}}

\begin{document}

\title{Long time behaviour of Ricci flow on open 3-manifolds}

\author{Laurent Bessi\`eres \thanks{Institut de Math\'ematiques de Bordeaux, universit\'e de Bordeaux.} \and  G\'erard Besson \thanks{Institut Fourier-CNRS, 
 universit\'e de Grenoble.} \and Sylvain Maillot \thanks{Institut de Math\'ematiques et de mod\'elisation de Montpellier, universit\'e de Montpellier.}}


\maketitle

\begin{abstract}
We study the long time behaviour of the Ricci flow with bubbling-off on a possibly noncompact $3$-manifold of finite volume whose universal cover has bounded geometry. As an application, we give a Ricci flow proof of Thurston's hyperbolisation theorem for $3$-manifolds with 
toral boundary that generalises Perelman's proof of the hyperbolisation conjecture in the closed case.
\end{abstract}

\thanks{This research was partially supported by ANR project GTO ANR-12-BS01-0014.}

\section{Introduction}

A Riemannian metric is \bydef{hyperbolic} if it is complete and has constant sectional curvature equal to $-1$. If $N$ is a $3$-manifold-with-boundary, then we say it is \bydef{hyperbolic} if its interior admits a hyperbolic metric.
In the mid-1970s, W.~Thurston stated his \emph{Hyperbolisation Conjecture,} which gives a natural  sufficient condition on the topology 
of a $3$-manifold-with-boundary $N$ which implies that it is hyperbolic. Recall that $N$ is \bydef{irreducible} if every embedded $2$-sphere in $N$ 
bounds a $3$-ball. It is \bydef{atoroidal} if every incompressible embedded $2$-torus in $N$ is parallel to a component of $\bord N$ or bounds a
product neighbourhood $T^2\times [0,1)$ of an end of $N$.
A version of Thurston's conjecture states that if $N$ is compact, connected, orientable, irreducible, and $\pi_1N$ is infinite and 
does not have any subgroup isomorphic to $\Zz^2$, then $N$ is hyperbolic. If one replaces the hypotheses on the fundamental group by the 
assumption that $N$ is atoroidal then one gets the conclusion that $N$ is hyperbolic or Seifert fibred.

Thurston proved his conjecture for the case of so-called \emph{Haken manifolds}, which includes the case where $\bord N$ is nonempty. 
The  case where $N$ is closed was solved by G.~Perelman \cite{Per1,Per2} using Ricci flow with surgery, based on ideas of R.~Hamilton.

It is natural to ask whether the Hamilton-Perelman approach works when $\bord N\neq\emptyset$. The interior $M$ of $N$, on which one wishes 
to construct a hyperbolic metric, is then noncompact. This question can be divided into two parts: first, is it possible to construct some 
version of Ricci flow with surgery on such an open manifold $M$, under reasonable assumptions on the initial metric? Second, does it 
converge (modulo scaling) to a hyperbolic metric? A positive answer to both questions would give a Ricci flow proof of the full 
Hyperbolisation Conjecture logically independent of Thurston's results.

A positive answer to the first question was given in~\cite{bbm:openflow}, for initial metrics of bounded geometry, i.e. of bounded curvature 
and positive injectivity radius. If one considers irreducible manifolds, surgeries are topologically trivial: each surgery sphere bounds 
a $3$-ball. Hence a surgery splits off a $3$-sphere. In this situation we can refine the construction of the Ricci flow with surgery so that it is not 
necessary to perform the surgery topologically. We obtain a solution which is a piecewise smooth Ricci flow on a fixed manifold; at singular times, one performs only a metric surgery, changing the metric on some $3$-balls. This construction was defined in \cite{B3MP} in the case of closed 
irreducible nonspherical $3$-manifolds, and called Ricci flow with bubbling-off. One can extend it to the setting of bounded geometry. The purpose  of this paper is to answer the second question, in the situation where the initial metric has a \emph{cusp-like structure}.

\begin{defi}\label{def:cusp-like} 
We say that a metric $g$ on $M$ has a \bydef{cusp-like structure}, or is a \emph{cusp-like metric}, if $M$ has finitely many ends (possibly zero), and each end has a neighbourhood which admits a
metric $g_\cusp$ homothetic to a rank two cusp neighbourhood of a hyperbolic manifold 
such that $g-g_\cusp$ goes to zero at infinity in $C^k$-norm for all positive integers $k$.
(Thus if $M$ is closed, any metric is cusp-like.)
\end{defi}

 Note that such a metric is automatically complete with bounded curvature and of finite volume, but its injectivity radius 
equals zero  hence it does not have bounded geometry. However, except in the case where 
$M$ is homeomorphic to a solid torus, its universal covering does have bounded geometry (see Lemma~\ref{lem:bd}). Since solid tori are Seifert fibred, we will 
assume that $M$ is not homeomorphic to a solid torus when necessary. Also note that if $M$ admits a cusp-like metric, then $M$ admits a manifold compactification whose boundary is empty or a union of $2$-tori. This compactification is irreducible (resp.~atoroidal, resp.~Seifert-fibred) if and only if $M$ is irreducible (resp.~atoroidal, resp.~Seifert-fibred).

In section  \ref{sec:openflow}  we construct a Ricci flow with bubbling-off  on $M$, for any cusp-like initial metric, by passing to the universal cover and working equivariantly. For simplicity we restrict ourselves to the case where $M$ is nonspherical. This is not a problem since spherical manifolds are Seifert fibred. We also 
prove that the cusp-like structure is preserved by this flow (cf.~Theorem \ref{thm:cusp-like}).

Using this tool, we can adapt Perelman's proof of geometrisation to obtain the following result:

\begin{theo}\label{thm:geometrisation}
Let $M$ be a connected, orientable, irreducible, atoroidal $3$-manifold and $g_0$ be a metric on $M$ which is cusp-like 
at infinity.  Then $M$ is Seifert-fibred, or there exists a Ricci flow with bubbling-off $g(\cdot)$ on $M$ defined on $[0,\infty)$, such that $g(0)=g_0$, and as $t$ goes to infinity, $t^{-1}g(t)$ converges smoothly in the pointed topology for appropriate base points to some finite volume hyperbolic metric on $M$. Moreover, $g(\cdot)$ has finitely many singular times, and there are positive constants 
$T,C$ such that $|\Rm| < Ct^{-1}$ for all $t \ge T$.
\end{theo}

\label{rem:geometrisation}

%


If $N$ is a compact, connected, orientable $3$-manifold such that $\bord N$ is empty or a union of $2$-tori, then $M=\Int N$ always carries a 
cusp-like at infinity metric. Thus we obtain:

\begin{corol}[Thurston, Perelman]\label{corol:geometrisation}
Let $N$ be a compact, connected, orientable $3$-manifold-with-boundary such that $\bord N$ is empty or a union of $2$-tori. If $N$ is irreducible and atoroidal, then $N$ is Seifert-fibred or hyperbolic.
\end{corol}

Note that it should be possible to obtain this corollary directly from the closed case by a doubling trick. 
The point of this paper is to study the behaviour of Ricci flow in the noncompact case.\\

Let us review some results concerning global stability or convergence to finite 
volume hyperbolic metrics. In the case of surfaces, R. Ji, L. Mazzeo and N. Sesum \cite{Ji-Maz-Ses:cusps} show that if $(M,g_0)$ is complete, 
asymptotically hyperbolic of finite area with $\chi(M) < 0$, then the normalised Ricci flow with initial condition $g_0$ 
converges  exponentially to the unique complete hyperbolic metric in its conformal class. G. Giesen and P. Topping 
\cite[Theorem 1.3]{Gie-Top:incomplete} show that if $g_0$, possibly incomplete and with unbounded curvature, is in the 
conformal class of a complete finite area hyperbolic metric $g_\hyp$, then there exists a unique Ricci flow with 
initial condition $g_0$ which is instantaneously complete and maximaly stretched (see the precise definition in \cite{Gie-Top:incomplete}), defined on $[0,+\infty)$ and 
such that the rescaled solution $(2t)^{-1}g(t)$ converges smoothly locally to $g_0$ as $t \to \infty$. Moreover, 
if $g_0 \leq C g_{\hyp}$ for some constant $C>0$ then the converence is global: 
for any $k \in \NN$ and $\mu \in (0,1)$ there exists a constant $C>0$ such that for all $t \geq 1$, 
$\vert (2t)^{-1}g(t) - g_\hyp \vert_{C^k(M,g_\hyp)} < \frac{C}{t^{1-\mu}}$. In dimensions greater than or equal to 3, R.~Bamler \cite{Bam:stability} shows that if $g_0$ is a small $C^0$-perturbation of a 
complete finite volume  hyperbolic metric $g_{\hyp}$, that is if $\vert g_0 - g_\hyp \vert_{C^0(M,g_\hyp)} <\epsi$ 
where $\epsi = \epsi(M,g_{hyp}) >0$, then the normalised Ricci flow with initial condition $g_0$ is defined 
for all time and converges in the pointed Gromov-Hausdorff topology to $g_{hyp}$. In dimension $3$ at least, there cannot 
be any global convergence result. Indeed, consider a complete finite volume hyperbolic manifold $(M^3,g_\hyp)$ with at least 
one cusp. Let $g_0$ be a small $C^0$ pertubation of $g_\hyp$ such that $g_0$ remains 
cusp-like at infinity but with a different hyperbolic structure  in the given cusp (change the cross-sectional flat 
structure on the cusp). By Bamler \cite{Bam:stability} a rescaling of $g(t)$ converges in the pointed topology to 
$g_{\hyp}$. The pointed convergence takes place on balls of radius $R$ for all $R$;  however, our stability theorem \ref{thm:stability2} implies that, out of these balls, the cusp-like structure   of 
$g_0$ is preserved for all time, hence is different from the one of the pointed limit. The convergence cannot be global.\\

The paper is organised as follows. In Section 2 we introduce the necessary definitions and we prove 
the existence of a Ricci flow with bubbling-off which preserves cusp-like structures. Section 3 is devoted to a thick-thin decomposition theorem  which shows that the thick part of 
$(M,t^{-1}g(t))$ (sub)-converges to a complete finite volume hyperbolic manifold. We give also some estimates on the long time behaviour of our solutions. 
In Section 4 we prove the incompressibility of the tori bounding the thick part. Section 5 is devoted to a collapsing theorem, which is used to 
show that the thin part is a graph manifold. 
Finally the main theorem \ref{thm:geometrisation} is proved in Section 6. To obtain the curvature estimates on the thin part, we follow 
\cite{Bam:longtimeI}. An overview of the proof is given at the beginning of that section.\\

Throughout this paper, we will use the following convention: \emph{all $3$-manifolds  are connected and orientable.}

Finally, we acknowledge the support of the Agence Nationale de la Recherche through Grant ANR-12-BS01-0004.

\section{Ricci flow with bubbling-off on open manifolds}\label{sec:openflow}

\subsection{Definition and existence}

In this section we define Ricci flow with bubbling-off and state the main existence theorem.

For convenience of the reader we recall here the most important definitions involved, and refer 
to Chapters~2, 4, and~5 of the monograph~\cite{B3MP} for completeness.  

\begin{defi}[Evolving metric]
Let $M$ be an $n$-manifold and $I\subset\Rr$ be an interval. An \bydef{evolving metric} on $M$ defined on $I$ is 
a map $t\mapsto g(t)$ from $I$ to the space of smooth Riemannian metrics on $M$. A \bydef{regular} time is a 
value of $t$ such that this map is $C^1$-smooth in a neighbourhood of $t$. If $t$ is not regular, 
then it is \bydef{singular}. We denote by $g_+(t)$ the right limit of $g$ at $t$, when it exists. 
An evolving metric is \bydef{piecewise $C^1$} if singular times form a discrete subset of $\RR$ and if $t \mapsto g(t)$ 
is left continuous and has a right limit at each point. A subset $N \times J \subset M\times I$ is 
\emph{unscathed} if $t \to g(t)$ is smooth there. Otherwise it is \emph{scathed}.
\end{defi}

If $g$ is a Riemannian metric, we denote by $\Rmin(g)$ (resp.~$\Rmax(g)$) the infimum (resp.~the supremum) of the scalar curvature of $g$. 
For any $x \in M$, we denote by $\Rm(x) : \Lambda^2 T_xM \to \Lambda^2 T_xM$ the \emph{curvature operator} defined by 
$$\langle \Rm(X \wedge Y),Z \wedge T \rangle = \Riem(X,Y,Z,T),$$
where $\Riem$ is the Riemann curvature tensor and $\wedge$ and $\langle \cdot, \cdot \rangle$ are normalised so that 
 $\{e_i \wedge e_j \mid i<j\}$ is an orthonormal basis if $\{e_i\}$ is. In particular, if $\lambda \geq \mu \geq \nu$ are the eigenvalues 
of $\Rm$, then $\lambda$ (resp. $\nu$) is the maximal (resp. minimal) sectional curvature  and $R=2(\lambda+\mu+\nu)$. \footnote{This convention is different from that used by Hamilton and other authors.}

\begin{defi}[Ricci flow with bubbling-off]
A piecewise $C^1$ evolving metric $t\mapsto g(t)$ on $M$ defined on $I$ is a \bydef{Ricci flow with bubbling-off} if
\begin{enumerate}[(i)]
\item The Ricci flow equation ${\partial g\over\partial t} = -2\Ric$ is satisfied at all regular times;
\item for every singular time $t\in I$ we have
     \begin{enumerate}[(a)]
        \item $\Rmin(g_+(t)) \geqslant \Rmin(g(t))$, and
      \item $g_+(t) \leqslant g(t)$.
     \end{enumerate}
\end{enumerate}
\end{defi}

\begin{rem}\label{rem:finite volume}
If $g(\cdot)$ is a complete Ricci flow with bubbling-off of bounded sectional curvature defined on an interval of type $[0,T]$ or $[0,\infty)$, and if $g(0)$ has finite volume, then $g(t)$ has finite volume for every $t$.
\end{rem}

A \bydef{parabolic neighbourhood} of a point $(x,t)\in M\times I$ is a set of the form
$$P(x,t,r,-\Delta t) = \{ (x',t') \in M\times I \mid x'\in B(x,t,r), t'\in [t-\Delta t, t] \}.$$

\begin{defi}[$\kappa$-noncollapsing]
For $\kappa,r>0$ we say that $g(\cdot)$ is \bydef{$\kappa$-collapsed} at $(x,t)$ on the scale $r$ 
if for all $(x',t')$ in the parabolic neighbourhood $P(x,t,r,-r^2)$ we have $|\Rm (x',t')| \le r^{-2}$ and $\vol(B(x,t,r))<\kappa r^n$. 
Otherwise, $g(\cdot)$ is \bydef{$\kappa$-noncollapsed} at $(x,t)$ on the scale $r$.  If this is true for 
all $(x,t)\in M\times I$, then we say that $g(\cdot)$ is \bydef{$\kappa$-noncollapsed on the scale $r$}.
\end{defi}

Next is the definition of \emph{canonical neighbourhoods}. 
From now on and until the end of this section,  $M$ is a $3$-manifold and $\epsi,C$ are positive numbers.

\begin{defi}[$\epsi$-closeness, $\epsi$-homothety]
If $U\subset M$ is an open subset and $g, g_0$ are two Riemannian metrics on $U$ we say that $g$ is $\epsi$-close to $g_0$ on $U$ 
if $$ ||g-g_0||_{[\epsi^{-1},U,g_0]} < \epsi,$$
where the norm is defined on page 26 of~\cite{B3MP}. We say that $g$ is $\epsi$-homothetic to $g_0$ on $U$ if there exists 
$\lambda>0$ such that $\lambda g$ is $\epsi$-close to $g_0$ on $U$. 
A pointed Riemannian manifold $(U,g,x)$ is $\epsi$-close to another Riemannian manifold $(U_0,g_0,x_0)$ if there exists a 
$C^{[\epsi^{-1}]+1}$-diffeomorphism $\psi$ from $U_0$ to $U$ sending $x_0$ to $x$ and such that the pullback metric 
$\psi^\ast(g)$ is $\epsi$-close to $g_0$ on $U$. We say that $(U,g,x)$ is $\epsi$-homothetic to $(U_0,g_0,x_0)$ if there exists 
$\lambda>0$ such that $(U,\lambda g,x)$ is $\epsi$-close to $(U_0,g_0,x_0)$. 
\end{defi}

\begin{defi}[$\epsi$-necks, $\epsi$-caps]
Let $g$ be a Riemannian metric on $M$. If $x$ is a point of $M$, then an open subset  
$U\subset M$ is an \bydef{$\epsi$-neck centred at $x$} if $(U,g,x)$ is $\epsi$-homothetic to 
$(S^2\times (-\epsi^{-1},\epsi^{-1}),g_\mathrm{cyl},(*,0))$, where $g_\mathrm{cyl}$ is the 
standard metric with unit scalar curvature. An open set $U$ is an \bydef{$\epsi$-cap centred at $x$} if $U$ 
is the union of two sets $V,W$ such that $x\in\Int V$, $V$ is a closed $3$-ball, $\bar W\cap V=\bord V$, and 
$W$ is an $\epsi$-neck.
\end{defi}

\begin{defi}[$(\epsi,C)$-cap] \label{def:epsi-C cap}
 An open subset $U\subset M$ is an \bydef{$(\epsi,C)$-cap centred at $x$} if $U$ is an $\epsi$-cap centred at $x$ and satisfies the following estimates:
 $R(x)>0$ and there exists $r\in (C^{-1} R(x)^{-1/2},C R(x)^{-1/2})$ such that
\begin{enumerate}[(i)]
\item $\overline{B(x,r)} \subset U \subset B(x,2r)$; 
\item The scalar curvature function restricted to $U$ has values in a compact
subinterval of $(C^{-1} R(x),C  R(x))$;
\item $\vol(U) > C^{-1}R(x)^{-3/2}$ and if $B(y,s) \subset U$ satisfies $|\Rm| \leqslant s^{-2}$ on $B(y,s)$ then
$$ C^{-1} < \frac{\vol B(y,s)}{s^3} ~; $$
\item On $U$, 
$$ |\nabla R|< C R^{\frac32}~, $$
\item On $U$, 
\begin{equation}
|\Delta R + 2|\Ric|^2|< C R^2 ~, \label{eq:Delta R}
\end{equation}
\item On $U$, 
$$ |\nabla \Rm|< C |\Rm|^{\frac32}~,$$
\end{enumerate}
\end{defi}

\begin{rem}\label{rem:partial l} If $t \mapsto g(t)$ is a Ricci flow, then 
$\frac{\partial R}{\partial t} = \Delta R + 2|\Ric|^2$ hence equation \eqref{eq:Delta R} implies 
that $|\dfrac{\partial R}{\partial t}| \leq C R^2$.
 \end{rem}

\begin{defi}[Strong $\epsi$-neck]
We call \bydef{cylindrical flow} the pointed evolving manifold $(S^2\times\Rr,\{g_\mathrm{cyl}(t)\}_{t\in (-\infty,0]})$, where $g_\mathrm{cyl}(\cdot)$ is 
the product Ricci flow with round first factor, normalised so that the scalar curvature at time $0$ is $1$. If 
$g(\cdot)$ is an evolving metric on $M$, and $(x_0,t_0)$ is a point in spacetime, then an open subset $N\subset M$ is a \bydef{strong $\epsi$-neck 
centred at $(x_0,t_0)$} if there exists $Q>0$ such that $(N,\{g(t)\}_{t\in [t_0-Q^{-1},t_0]},x_0)$ is unscathed, and, denoting $\bar g(t)=Qg(t_0+tQ^{-1})$ the  
parabolic rescaling with factor $Q>0$ at time $t_0$, $(N,\{\bar g(t)\}_{t\in [-1,0]},x_0)$ is $\epsi$-close to $(S^2\times (-\epsi^{-1},\epsi^{-1}),\{g_\mathrm{cyl}(t)\}_{t\in [-1,0]},*)$.
\end{defi}

\begin{rem}
A strong $\epsi$-neck satisfies the estimates (i)--(vi) of Definition 
\ref{def:epsi-C cap} for an appropriate constant $C=C(\epsi)$, at all times, that is on all $N \times [t_0-Q^{-1},t_0]$ for any $Q>0$ as above.
\end{rem}

\begin{defi}[$(\epsi,C)$-canonical neighbourhood] \label{def:VC}
Let  $\{g(t)\}_{t \in I})$ be an evolving metric on $M$. We say that a point $(x,t)$ admits (or is centre of) an
\bydef{$(\epsi,C)$-canonical neighbourhood} if $x$ is centre of an $(\epsi,C)$-cap in $(M,g(t))$ or if $(x,t)$ is centre of a strong $\epsi$-neck $N$ which satisfies (i)--(vi) at all times.  
\end{defi}

In~\cite[Section 5.1]{B3MP} we fix constants $\epsi_0,C_0$. For technical reasons, we need to take them slightly different here; this will be explained in the proof of 
Theorem~\ref{thm:existence}.

\begin{defi}[Canonical Neighbourhood Property $(CN)_r$]\label{def:CN}
Let $r>0$. An evolving metric satisfies the property $(CN)_r$ if, for any $(x,t)$, if $R(x,t) \geq r^{-2}$ then $(x,t)$ is centre 
of an $(\epsi_0,C_0)$-canonical neighbourhood. 
\end{defi}

Next we define a pinching property for the curvature tensor coming from work of 
Hamilton and Ivey \cite{hamilton:nonsingular,Ive}. We consider a familly of positive functions $(\phi_{t})_{t\geqslant 0}$ defined as follows. Set  $s_{t}:=\frac{e^2}{1+t}$ and define 
$$\phi_{t} :[-2s_{t},+\infty) \longrightarrow [s_{t},+\infty)$$
\index{$\phi_{t}$}%
as the reciprocal of the increasing function $$s \mapsto 2s(\ln(s) + \ln(1+t) -3).$$
Compared with the expression used in \cite{hamilton:nonsingular,Ive}, there is an extra factor $2$ here. This comes from our curvature conventions. 
A key property of this function is that $\frac{\phi_{t}(s)}{s} \to 0$ as $s \to +\infty$. 
\begin{defi}[Curvature pinched toward positive]\label{def:pinching}
Let $I\subset [0,\infty)$ be an interval and $\{g(t)\}_{t\in I}$ be an evolving metric on  $M$. We say that $g(\cdot)$ has \bydef{curvature
pinched toward positive at time $t$}  if for all $x \in M$ we have
\begin{gather}
R(x,t)  \geqslant -\frac{6}{4t+1}, \label{eq:pinching 1} \\
\Rm (x,t) \geqslant -\phi_{t}(R(x,t)). \label{eq:pinching 2}
\end{gather}
We say that $g(\cdot)$ has  \bydef{curvature
pinched toward positive} if it has curvature
pinched toward positive at each $t \in I$. 
\end{defi}

This allows in particular to define the notion of \emph{surgery parameters} $r,\delta$ (cf.~\cite[Definition 5.2.5]{B3MP}). 
Using~\cite[Theorem~5.2.4]{B3MP} we also define their \emph{associated cutoff parameters} $h,\Theta$. Using the 
metric surgery theorem, we define the concept of a metric $g_+$ being \emph{obtained from $g(\cdot)$ by $(r,\delta)$-surgery at time $t_0$} (cf.~\cite[Definition 5.2.7]{B3MP}). This permits to define the following central notion:

\begin{defi}[Ricci flow with $(r,\delta)$-bubbling-off]\label{defi:r delta solution}
Fix surgery parameters $r,\delta$ and let $h,\Theta$ be
the associated cutoff parameters. Let $I\subset [0,\infty)$ be an interval and $\{g(t)\}_{t\in I}$ be a
 Ricci flow with bubbling-off on $M$. We say that $\{g(t)\}_{t\in I}$ is a 
\bydef{Ricci flow with $(r,\delta)$-bubbling-off} if it has the following properties:
\begin{enumerate}[(i)]
\item $g(\cdot)$ has curvature pinched toward positive and satisfies $R(x,t)\leqslant \Theta$ for all $(x,t)\in M\times I$;
\item For every singular time $t_0\in I$, the metric $g_+(t_0)$ is obtained 
from $g(\cdot)$ by $(r,\delta)$-surgery at time $t_0$; 
\item  $g(\cdot)$ satisfies property $(CN)_r$. 
\end{enumerate}
\end{defi}

\begin{defi}[Ricci flow with $(r,\delta,\kappa)$-bubbling-off]
Let $\kappa>0$. A Ricci flow with $(r,\delta)$-bubbling-off $g(\cdot)$  
is called a \bydef{Ricci flow with $(r,\delta,\kappa)$-bubbling-off} if it is $\kappa$-noncollapsed on all scales less than or equal to 1.
\end{defi}

\begin{defi}\label{def:normalised}
A metric $g$ on a $3$-manifold $M$ is \bydef{normalised} if it satisfies $\operatorname{tr} \Rm^2\le 1$ and each ball of 
radius $1$ has volume at least half of the volume of the unit ball in Euclidean $3$-space.
\end{defi}

Note that a normalised metric always has bounded geometry.
At last we can state our existence theorem:
\begin{theo}\label{thm:existence}
There exist decreasing sequences of positive numbers $r_k,\kappa_k>0$ and, for every continuous positive function $t\mapsto \bar \delta(t)$, 
a decreasing sequence of positive numbers $\delta_k$ with $\delta_k \leqslant \bar \delta(\cdot)$ on $]k,k+1]$ with the following property.
For any complete, normalised, nonspherical, irreducible Riemannian $3$-manifold $(M,g_0)$, one of the following conclusions holds:
\begin{enumerate}[(i)]
\item There exists $T>0$ and a complete Ricci flow with bubbling-off $g(\cdot)$ of bounded geometry on $M$, defined on $[0,T]$, with $g(0)=g_0$, and such that every point of $(M,g(T))$ is centre of an $\epsi_0$-neck or an $\epsi_0$-cap, or
\item \label{item:stop} There exists a complete Ricci flow with bubbling-off $g(\cdot)$ of bounded geometry on $M$, defined on 
$[0,+\infty)$, with $g(0)=g_0$, and such that for every nonnegative integer $k$, the restriction of $g(\cdot)$ to 
$]k,k+1]$ is a Ricci flow with $(r_k,\delta_k,\kappa_k)$-bubbling-off.
\end{enumerate}
\end{theo}

\begin{defi}[Ricci flow with $(r(\cdot),\delta(\cdot))$-bubbling-off]\label{def:rdelta}
 We fix forever a function $r(\cdot)$ such that $r(t)=r_k$ on each interval $]k,k+1]$. Given $\delta(\cdot)$ satisfying 
$\delta(t)=\delta_k$ on all $]k,k+1]$, we call a solution as above a \emph{Ricci flow with $(r(\cdot),\delta(\cdot))$-bubbling-off}. 
We define similarly $h(\cdot)$ and $\Theta(\cdot)$ their associated cutoff parameters.
\end{defi}

\begin{addendum}[Ricci flow with bubbling-off on the quotient]\label{add:existence}
With the same notation as in Theorem~\ref{thm:existence} and under the same hypotheses, if in addition $(M,g_0)$ is a Riemannian 
cover of some Riemannian manifold $(X,\bar g_0)$, then in either case there exists a Ricci flow with bubbling-off $\bar g(\cdot)$ on $X$ such that for each $t$, $(M,g(t))$ is a Riemannian cover of $(X,\bar g(t))$, and in Case~(ii), the restriction of $\bar g(\cdot)$ to $]k,k+1]$ is a Ricci flow with $(r_k,\delta_k)$-bubbling-off for every $k$.
\end{addendum}

The only differences between Theorem \ref{thm:existence} and Theorem 11.5 of~\cite{bbm:openflow} is that $M$ is assumed to be irreducible, that `surgical solution' is replaced with `Ricci flow with bubbling-off', and that there is the alternative conclusion (i).

Theorem~\ref{thm:existence}  follows from iteration of the following result, which is analogous to~\cite[Theorem 5.6]{bbm:openflow}:

\begin{theo}\label{thm:a iterer}
For every $Q_0,\rho_0$ and all $0\le T_A<T_\Omega<+\infty$, there exist $r,\kappa>0$ and for all $\bar \delta>0$ there exists 
$\delta \in (0,\bar \delta)$ 
with the following property. For any complete, nonspherical, irreducible Riemannian $3$-manifold $(M,g_0)$ which satisfies $|\Rm| \le Q_0$, has injectivity radius at least $\rho_0$, has curvature pinched toward positive at time $T_A$, one of the following conclusions holds:
\begin{enumerate}[(i)]
\item There exists $T\in (T_A,T_\Omega)$ and a Ricci flow with bubbling-off $g(\cdot)$ on $M$, defined on $[T_A,T]$, with $g(T_A)=g_0$, and such that every point of $(M,g(T))$ is centre of an $\epsi_0$-neck or an $\epsi_0$-cap, or
\item There exists a Ricci flow with $(r,\delta,\kappa)$-bubbling-off $g(\cdot)$ on $M$, defined on $[T_A,T_\Omega]$, satisfying $g(T_A)=g_0$.
\end{enumerate}
\end{theo}

The proof of Theorem~\ref{thm:a iterer} is the same as~\cite[Theorem 5.6]{bbm:openflow}. It follows from three propositions, which we do not write here, analogous to Propositions A, B, C of~\cite{bbm:openflow} (see the propositions page 949). 
The only notable difference is that we have to modify Proposition A to add the alternative conclusion that in $(M,g(b))$, every point is centre 
of an $\epsi_0$-cap or an $\epsi_0$-neck. Let us explain the proof of this adapted proposition A (see ~\cite{bbm:openflow} pages 959-961). It uses the surgical procedure of the monograph~\cite{B3MP} rather than that of~\cite{bbm:openflow}. If the curvature is large everywhere, that is if $R \geq 2r^{-2}$ on 
$(M,g(b))$ where $r$ is the surgery parameter, then by property $(CN)_r$ (Definitions 2.10 and 2.12 (iii)) every point has a canonical neighbourhood, so the alternative conclusion holds. Otherwise, we partition $M$ in three sets of small, large or very large curvature. Precisely, as in \cite[page 89]{B3MP}, we define $\mathcal{G}$ (resp.~$\mathcal{O}$, resp.~$\mathcal{R}$) as the set of points of $M$ of scalar curvature less than $2 r^{-2}$, (resp.~$\in [2r^{-2}, \Theta /2)$, resp.~$\geqslant \Theta /2$). By the assumption that 
$\Rmin(b) < 2r^{-2}$ and $\Rmax(b)=\Theta$, these sets are nonempty. One can find a locally finite collection of cutoff $\delta$-necks $\{N_i\}$ in $\mathcal{O}$ which separates $\mathcal{G}$  
from $\mathcal{R}$, in the sense that any connected component of $M \setminus \{N_i\}$ is contained in $\mathcal{G} \cup \mathcal{O}$  or in 
$\mathcal{O} \cup \mathcal{R}$. Since $M$ is irreducible and not homeomorphic to $S^3$, the middle sphere of each $N_i$ bounds a unique topological $3$-ball $B_i$. Then one of the following cases occurs:

\paragraph{Case 1} Each $B_i$ is contained in a unique maximal $3$-ball $B_j$.
  
If $\mathcal{O}$ is contained in the union of maximal $B_j$'s, we can perform the surgical procedure using the Metric surgery theorem 5.2.2 of~\cite{B3MP} on each maximal cap $B_j$, yielding a metric which has the desired properties. Otherwise 
one can see that each point of $M$ is centre of $\epsi$-cap. Hence the alternative conclusion holds.  

\paragraph{Case 2} $M$ is the union of the $B_i$'s.

Then each point is separated from infinity by a cutoff neck, so each point is centre of a cap. Hence the alternative conclusion holds.\\

Finally, we need to explain how the addendum is proved. We already remarked in~\cite{bbm:openflow} Section 11 that the construction can be made equivariant with respect to a properly discontinuous group action, by work of Dinkelbach and Leeb~\cite{dl:equi}. The only thing to check is that we still have the Canonical Neighbourhood Property for the quotient evolving metric $\bar g(\cdot)$. This is not obvious, since the projection map $p:M\to X$ might not be injective when restricted to a canonical neighbourhood.

We use a classical trick: by adjusting the constants, we may assume that $g(\cdot)$ has the stronger property that each point $(x,t)$ such 
that $R(x,t)\ge r^{-2}$ has an $(\epsi_0/2,C_0)$-canonical neighbourhood. Take now $(x,t)\in X\times I$ such that $R(x,t)\ge r^{-2}$. 
Choose $\bar x\in M$ such that $p(\bar x)=x$. Then $R(\bar x,t)=R(x,t)\ge r^{-2}$, so $(\bar x,t)$ has an $(\epsi_0/2,C_0)$-canonical 
neighbourhood $U$. By truncation, it also has an $(\epsi_0,C_0)$-canonical neighbourhood $U'$ contained in $U$ (see figure below) : 

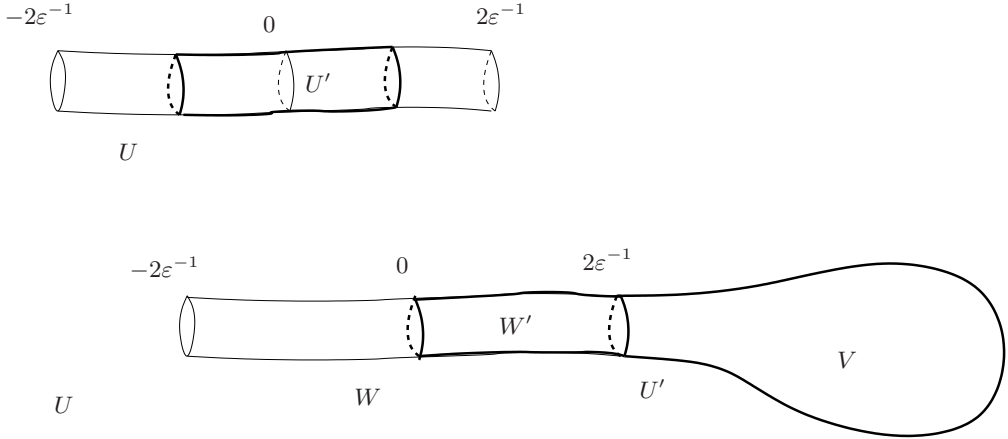
\begin{center}
 \input{truncation.pstex_t}
\end{center}

Precisely, 
if $U$ is an $\epsi_0/2$-neck with parametrisation $\phi : S^2\times (-2\epsi_{0}^{-1},2\epsi_{0}^{-1}) \to U$, we set 
$U':=\phi(S^2\times (-\epsi_{0}^{-1},\epsi_{0}^{-1}))$. If $U$ is a cap, then $U$ is the union of two sets $V,W$, where $\overline{W} \cap V=\bord V$ and $W$ is an $\epsi_0/2$-neck with parametrisation $\phi$. Then we set $W':= \phi(S^2\times (0,2\epsi_{0}^{-1}))$ and $U':=V \cup W'$.

\begin{claim}
The restriction of the projection map $p$ to $U'$ is injective.
\end{claim}

Once the claim is proved, we can just project $U'$ to $X$ and obtain an $(\epsi_0,C_0)$-canonical neighbourhood for $(x,t)$, so we are done.

To prove the claim we consider two cases:

\paragraph{Case 1} $U$ and $U'$ are caps.

Assume by contradiction that there is an element $\gamma$ in the deck transformation group, different from the identity, and a 
point $y\in U'$ such that $\gamma y\in U'$. Following~\cite{dl:equi}, we consider the subset $N_{\epsi_0}$ of $M$ consisting of 
points which are centres of $\epsi_0$-necks. According to~\cite[Lemmas 3.6, 3.7]{dl:equi} there is an open supset $F \supset \overline{N_{\epsi_0}}$ which has an equivariant foliation $\mathcal F$ by almost round $2$-spheres. 
All points sufficiently close to the centre of $W$ are centres of $\epsi_0$-necks.

Pick a point $z$ in $N_{\epsi_0} \cap W\setminus W'$ sufficiently far from $W'$ so that the leaf $S$ of $\mathcal F$ through $z$ is disjoint from $U'$. 
By Alexander's theorem, $S$ bounds a $3$-ball $B\subset U$. Note that $B$ contains $U'$. If $S=\gamma S$, then $B=\gamma B$ or $M=B\cup \gamma B$. The former possibility is ruled out by the fact that the action is free, while any self-homeomorphism of the $3$-ball has a fixed point. The latter is ruled out by the assumption that $M$ is not diffeomorphic to $S^3$.

Hence $S\neq \gamma S$. Since $S$ and $\gamma S$ are leafs of a foliation, they are disjoint. Then we have the following three possibilities:

\paragraph{Subcase a} $\gamma S$ is contained in $B$.

Then we claim that $\gamma B\subset B$. Indeed, otherwise we would have $M=B \cup \gamma B$, and $M$ would be diffeomorphic to $S^3$. Now $\gamma$ acts by isometry, so $\vol B=\vol \gamma B$. This is impossible since the annular region between $S$ and $\gamma S$ has nonzero volume. 

\paragraph{Subcase b} $S$ is contained in $\gamma B$.
This case is ruled out by a similar argument exchanging the roles of $S$ and $\gamma S$ (resp.~of $B$ and $\gamma B$.)

\paragraph{Subcase c}  $B$ and $\gamma B$ are disjoint.

Then since $U'\subset B$, the sets $U'$ and $\gamma U'$ are also disjoint, contradicting the existence of $y$.

\paragraph{Case 2} $U$ and $U'$ are necks.
Seeking a contradiction, let $\gamma$ be an element of the deck transformation group, different from the identity, and  $y$ be a point of $U'$ such that $\gamma y\in U'$. Consider again the set $N_{\epsi_0}$ defined above and the equivariant foliation $\mathcal F$. Since $U'$ is contained in the bigger set $U$, each point of $U'$ is centre of an $\epsi_0$-neck. Let $S$ (resp.~$\gamma S$) be the leaf of $\mathcal F$ passing through $y$ (resp.~$\gamma y$.) Since $M$ is irreducible, $S$ (resp.~$\gamma S$) bounds a $3$-ball $B$ (resp.~$B_\gamma$). As in the previous case, we argue that one of these balls is contained into the other, otherwise we could cover $M$ by $B,B_\gamma$ and possibly an annular region between them, and get that $M$ is diffeomorphic to $S^3$. Since $\gamma$ acts by an isometry, we must in fact have $B=B_\gamma$, and $\gamma$ has a fixed point, contradicting our hypotheses. This finishes the proof of the claim, hence that of Addendum~\ref{add:existence}.

\subsection{Stability of cusp-like structures}

In this section, we prove the stability of cusp-like structures under Ricci flow with bubbling-off. We consider a (nonspherical, irreducible) $3$-manifold $M$, endowed with a cusp-like metric $g_0$. 
To begin we remark that the universal cover of $M$ has bounded geometry, except in the case 
of solid tori:

\begin{lem}\label{lem:bd}
Assume that $M$ is not homeomorphic to a solid torus. 
Let $(\tilde M,\tilde g_0)$ denote the universal cover of $(M,g_0)$. Then $(\tilde M,\tilde g_0)$ has bounded geometry.
\end{lem}

\begin{proof}
Sectional curvature is bounded on $(M,g_0)$, hence on the universal cover $(\tilde M,\tilde g_0)$ by the same constant. 
Observe that for any lift $\tilde x \in \tilde M$ of some $x\in M$, the injectivity radius at $\tilde x$ is not less than the injectivity radius 
at $x$. Fix a compact subset $K \subset M$  such that each connected component $C$ of $M \setminus K$ is $\epsi$-homothetic 
to a hyperbolic cusp neighbourhood, for some small $\epsi>0$. Let $\tilde K$ denote any lift of $K$ to $\tilde M$. Then the $5$-neighbourhood of $\tilde K$ 
has injectivity radius bounded below by $i_0>0$, the injectivity radius of the (compact) $5$-neighbourhood of $K$. Now consider a lift $\tilde C$ of a 
cuspidal component $C$. The boundary $\partial C$ is incompressible in $M$, otherwise $M$ would be homeomorphic to a solid 
torus (see Theorem A.3.1 in \cite{B3MP}). It follows that $\tilde C$ is simply connected with an incomplete metric of negative sectional curvature. 
Arguing as in the proof of the Hadamard theorem, it follows that the injectivity radius at a given point $p \in \tilde C$ is not less than 
$d(p,\partial \tilde C)$. 
Together with the previous estimate, this implies that $\inj(\tilde M,\tilde g_0) \ge \min\{i_0,5\}>0$.  
\end{proof}

Let us denote by $g_c$ a metric on $M$ which is hyperbolic on the complement of some compact subset of $M$, and such that, 
for each end $E$ of $M$ there is a factor $\lambda_E>0$ such that $\lambda_E g_0 - g_c$ goes to zero at infinity in the end, 
in $C^k$-norm for each integer $k$. Let $g(\cdot)$ be a Ricci flow with $(r(\cdot),\delta(\cdot))$-bubbling-off on $M$ such that $g(0)=g_0$, 
defined on $[0,T]$ for some $T>0$. Set $\lambda_E(t)=\frac{\lambda_E}{1+4\lambda_Et}$. We then have:
\begin{theo}[Stability of cusp-like structures]\label{thm:cusp-like} \label{thm:stability2}  
For each end $E$ of $M$, $\lambda_E(t) g(t)-g_c$ goes to zero at infinity in this end, in $C^k$-norm for each integer $k$, uniformly for $t \in [0,T]$.
\end{theo}

\begin{proof}
 
Let us first explain the idea. It is enough to work on each cusp. The main tool is the Persistence Theorem 8.1.3 from \cite{B3MP}, which 
proves that a Ricci flow remains close, on a parabolic neighbourhood where it has a priori curvature bounds, to a 
given Ricci flow model, if the initial data are sufficiently close on some larger balls. The model we use now is a hyperbolic 
Ricci flow on $T^2 \times \RR$. To obtain the required curvature bounds, we shall consider an interval 
$[0,t]$ where the closeness to the hyperbolic flow holds, and $\sigma>0$ fixed small enough 
so that Property $(CN)_r$, which prevents scalar curvature to explode too fast, gives 
curvature bounds on $[0,t+\sigma]$. The Persistence Theorem then gives closeness to the hyperbolic flow until time $t+\sigma$ on a smaller 
neighbourhood of the cusp. One can iterate this procedure, shrinking the neighbourhood of the cusp by a definite amount 
at each step, until time $T$.


Let us now give the details. Let $E$ be an end of $M$ and $U$ be a neighbourhood of $E$ such that $(U,g_c)$ is isometric 
to $(\textbf{T}^2\times [0, +\infty), g_\hyp=e^{-2r}g_{\textbf{T}^2}+dr^2)$, where $g_{\textbf{T}^2}$ is flat. 
Let $\phi: \textbf{T}^2\times [0, +\infty) \to U$ be an isometric parametrisation (between $g_c$ and $g_\hyp$.) Then  
$\lambda_E \phi^\ast g_0 - g_\hyp$ and its derivatives go to zero at infinity. 
We may assume for simplicity that $\lambda_E=1$, and we define $\bar g(t):=\phi^\ast g(t)$ to be the pullback Ricci flow with bubbling-off 
on $\textbf{T}^2\times [0, +\infty)$. Let $g_\hyp(\cdot)$ denote the Ricci flow on $\textbf{T}^2\times \textbf{R}$ such that $g_\hyp(0)=e^{-2r}g_{\textbf{T}^2}+dr^2$, i.e. $g_\hyp(t)=(1+4t)g_\hyp$. We use it as the Ricci flow model, in the sense of \cite[Theorem 8.1.3.]{B3MP}.  Our goal is to compare $g_\hyp(t)$ to $\bar g(t)$. 

By definition of our Ricci flow with bubbling-off, $r(\cdot)$ and $\Theta(\cdot)$ are piecewise constant. 
More precisely, there exist $0=t_0<t_1<\dots <t_N=T$ such that $r(t)=r_i$ and $\Theta (t)=\Theta_i$ on $(t_i, t_{i+1}]$. In fact, we can 
choose $t_i=i$ for $i<N$ (cf. Definition \ref{def:rdelta}).
In particular, $g(t)$ satisfies the canonical neighbourhood property at scale $r_{i}$  on this interval (every point at which the scalar curvature 
is greater than $r_i^{-2}$ is centre of an $(\varepsilon_0, C_0)$ canonical neighbourhood) and the 
scalar curvature is bounded above by $\Theta_i$. The pinching assumption (cf.~Definition~\ref{def:pinching}) then implies that the full curvature tensor is bounded by some $K_i$ on the same interval.

Set $K:=\sup_{i=1,\dots,N-1}\big\{ K_i\big\}$. Define a small number $\sigma >0$ by setting	
$$\sigma:=\frac{r_{N-1}^2}{2C_0}\leq \frac{r_i^2}{2C_0}\,\quad \forall i=0,\dots ,N-1\,.$$
This number is small enough so that $g(\cdot)$ cannot develop a singularity on a cusp on $[t,t+\sigma]$ if $R \leq 0$ at time $t$. Precisely, let us put ${\mathcal C}_s := \textbf{T}^2\times [s, +\infty)$, for $s \geq 0$. Then we have: 
\begin{lem}\label{lem:boundedcurv}
 If $\bar g(\cdot)$ is unscathed on ${\mathcal C}_s\times [0, \Delta ]$ and has scalar curvature $R\leqslant 0$ there, 
then it is also unscathed on ${\mathcal C}_s\times [0, \Delta +\sigma]$ and has curvature tensor bounded by $K$.
\end{lem}
\begin{proof}
We know that singular times are discrete. Let $t \in [0,\sigma]$ be maximal such that $\mathcal{C}_s\times [0, \Delta +t]$ is 
unscathed for $\bar g(\cdot)$ (possibly $t=0$). 

We prove first that for $x\in \mathcal C_s$ and  $t'\in [\Delta, \Delta +t]$  we have 
\begin{equation*}  
R(x,t')\leq 2r(t')^{-2}<<h(t')^{-2}.\,
\end{equation*}
Indeed, since $r(\cdot)$ is nonincreasing, $g(\cdot)$ satisfies $(CN)_{r(\Delta +t)}$ on $[\Delta,t']$. If $R(x,\Delta)\leq 0$ and $R(x,t')>2r(t')^{-2}$, then we can find a subinterval $[t_1,t_2]\subset [\Delta, t']$ such that for $u\in [t_1,t_2]$, 
$R(x, u)\geq r(t')^{-2}$, $r(x,t_1)=r(t')^{-2}$, and $r(x,t_2)=2r(t')^{-2}$.

Then the inequality $\vert \frac{\partial R}{\partial t}\vert <C_0 R^2$ holds on $\{x\}\times [t_{1},t_{2}]$, thanks to 
Property~\eqref{eq:Delta R} of canonical neighbourhoods (cf Remark \ref{rem:partial l}). 
The contradiction follows by integrating this inequality and using the fact that $t_2-t_1 <\sigma$.

Assume now that $t<\sigma $. Then there is a surgery at time $\Delta + t$ and, by definition of the maximal time, $\phi(\mathcal C_s)$ is scathed at time $\Delta +t$. The surgery spheres are disjoint 
from $\phi(\mathcal C_s)$, as they have curvature $\approx (h(\Delta+t))^{-2}$, where $h(\Delta+t)$ is the cutoff parameter, and curvature on $\phi(\mathcal C_s)$ is less than $2r(t')^{-2} <<(h(\Delta+t))^{-2}$. By definition of our surgery, this means that $\phi(\mathcal C_s) \subset M$ is contained in a $3$-ball where the metric surgery is performed. But a cusp of $M$ cannot be contained in a $3$-ball of $M$, hence we get a contradiction. We conclude that  $t=\sigma$ and $R(x,t')\leq 2r(t')^{-2}$, $\forall t'\in [\Delta , \Delta +\sigma ]$. The pinching assumption then implies $\vert \Rm \vert < K$ there.
\end{proof}


For every $A>0$, let $\rho_A =\rho (A,T,K)$ be given by the Persistence Theorem 8.1.3 of \cite{B3MP}. The proof of Theorem 
\ref{thm:stability2} is obtained by iteration of Lemma \ref{lem:boundedcurv} and the Persistence Theorem as follows.
%

Fix $A>0$. Let $s_0>0$ be large enough so that $\bar g(0)$ is $\rho_A^{-1}$-close to 
$g_\hyp(0)$ on $\calc_{s_0}$. In particular $R \leq 0$ there, so by
Lemma~\ref{lem:boundedcurv}, $\bar g(\cdot)$ is 
unscathed on $\calc_{s_0}\times [0,\sigma]$, with curvature tensor bounded by $K$. The above-mentioned Persistence 
Theorem applied to $P(q,0,A,\sigma )$, for all $q\in {\mathcal C}_{s_0+\rho_A}$, shows that 
$\bar g(t)$ is $A^{-1}$-close to $g_\hyp(t)$ there. Hence 
on ${\mathcal C}_{s_0+\rho_A-A}\times [0, \sigma ]$, $\bar g(\cdot)$ is $A^{-1}$-close to $g_\hyp(\cdot)$, and 
in particular $R\leqslant 0$ there. We then iterate this argument, applying Lemma~\ref{lem:boundedcurv} 
and the Persistence Theorem, $n=[T/\sigma]$ times and get that $\bar g(\cdot)$ is $A^{-1}$-close to 
$g_\hyp(\cdot)$ on $\calc_{s_0+n(\rho_A-A)}\times [0,T]$. 

By letting $A$ go to infinity and rescaling appropriately, this finishes the proof of Theorem~\ref{thm:stability2}.  
\end{proof}


\section{Thick-thin decomposition theorem}

Let $(X,g)$ be a complete Riemannian $3$-manifold and $\epsi$ be a positive number. The \bydef{$\epsi$-thin part} of $(X,g)$ is the subset $X^-(\epsi)$ of 
points $x\in X$ for which there exists $\rho\in (0,1]$ such that on the ball $B(x,\rho)$ all sectional curvatures are at least $-\rho^{-2}$ and the volume 
of this ball is less than $\epsi \rho^3$. Its complement is called the \bydef{$\epsi$-thick part} of $(X,g)$ and denoted by $X^+(\epsi)$. 
The aim of this section is to gather curvature and convergence estimates on the $\epsi$-thick part of $(M,t^{-1}g(t))$ as $t \to \infty$, when $g(\cdot)$ is a Ricci flow with $(r(\cdot),\delta(\cdot))$-bubbling-off for suitably chosen surgery parameters $r(\cdot)$ and $\delta(\cdot)$.
{\bf Here, we assume $M$ irreducible, nonspherical and not Seifert fibred. We assume also that $M$ is not homeomorphic to $\RR^3$}, which does not have cusp-like metrics.
As a consequence, {\bf M does not have a complete metric with $\Rm \ge 0$}. In the compact case, this follows from Hamilton's classification theorem
(Theorem~B.2.5 in Appendix~B of~\cite{B3MP}). In the noncompact case, this follows from the Cheeger-Gromoll theorem and the Soul theorem
(cf.~B.2.3 in~\cite{B3MP}).  

Recall that $r(\cdot)$ has been fixed in Definition~\ref{def:rdelta}. In \cite[Definition 11.1.4]{B3MP}, we define a positive nonincreasing function $\bar \delta(\cdot)$ such that any Ricci flow with $(r(\cdot),\delta(\cdot))$-bubbling-off satisfies some technical theorems---Theorems 11.1.3 and 11.1.6, analoguous to \cite[Propositions 6.3 and  6.8]{Per2}---if $\delta \leq \bar \delta$ and the initial metric is normalised.

Both Theorems~11.1.3 and~11.1.6 remain true for a Ricci flow with $(r(\cdot),\delta(\cdot))$-bubbling-off on a noncompact nonspherical irreducible manifold, with the weaker assumption that the metric has normalised curvature at time $0$, i.e. $\tr \Rm^2 \leqslant 1$ for the initial metric, instead of being normalised in the sense of Definition~\ref{def:normalised}. In particular it applies to metrics which are cusp-like at infinity. Indeed, the proofs of theorems 11.1.3 and 11.1.6 do not use the assumption on the volume of unit balls for the initial metric; it only uses the assumption on the curvature, mainly through the estimates \eqref{eq:pinching 1}-\eqref{eq:pinching 2}. It uses neither the compactness of the manifold, the finiteness of the volume nor the particular manifold. We recall that the core of Theorem~11.1.3 is to obtain $\kappa$-noncollapsing property, canonical neighbourhoods and curvature controls relatively to a distant ball satisfying a lower volume bound assumption. The parameters then 
depend on the distance to the ball and on its volume, not on time or initial data. These estimates are then used to control the thick part (Theorem 11.1.6).  

We gather below results following mainly from Perelman \cite[6.3, 6.8, 7.1-3]{Per2}. We need some definitions.

Given a Ricci flow with bubbling-off on $M$, we define
$$ \rho(x,t):=\max \{\rho >0\ :\  \Rm \geqslant -\rho^{-2}\ \textrm{ on } B(x,t,\rho)\ \}$$
and $\rho_{\sqrt{t}} := \min\{\rho(x,t),\sqrt{t}\}$. We denote by $\tilde M$ the universal cover of $M$ and $\tilde g(t)$ the lifted evolving metric, which is by Addendum \ref{add:existence} a Ricci flow with $(r(\cdot),\delta(\cdot))$-bubbling-off if $g(t)$ is. If $x \in M$, we denote by $\tilde x \in \tilde M$ a lift of $x$ and by $\tilde B(\tilde x,t,r)$ the $r$-ball in $(\tilde M,\tilde g(t))$ centered at $\tilde x$. An evolving metric $\{g(t)\}_{t\in I}$ on $M$ is said to have \bydef{finite volume} if $g(t)$ has finite volume for every $t\in I$. We denote this volume by $V(t)$. We then have:

\begin{prop}\label{prop:thick thin} \label{prop:bar rho bis} \label{thm:thick thin} For every $w>0$ there exists 
$0 < \bar \rho(w) <\bar r(w) < 1$, $\bar T=\bar T(w)$, $\bar K=\bar K(w) >0$ such that for any Ricci flow with $(r(\cdot),\delta(\cdot))$-bubbling-off $g(\cdot)$ on $M$ such that $\delta(\cdot)\le \bar\delta(\cdot)$ 
and with normalised curvature at time $0$, the following holds:
\begin{enumerate}[(i)]
\item \label{prop:bar r} For all $x \in M$, $t\geqslant \bar T$ and $0 < r \le \min\{\rho(x,t),\bar r \sqrt{t}\}$, if $\vol \tilde B(\tilde x,t,r) \geqslant wr^3$ for some 
lift $\tilde x$ of $x$ then $|\Rm| \leqslant Kr^{-2}$, $|\nabla \Rm| \leqslant Kr^{-3}$ and  $|\nabla^2 \Rm | \leqslant Kr^{-4}$ on $B(x,t,r)$.
\item  For all $x \in M$ and $t \geqslant \bar T$, if $\vol \tilde B(\tilde x,t,r) \ge wr^3$ for some lift 
$\tilde x$ of $x$ where $r=\rho(x,t)$, then 
$\rho(x,t) \geqslant \bar \rho \sqrt{t}$.
\item  If $g(\cdot)$ has finite volume, then:
\begin{enumerate}[(a)]
\item There exists $C>0$ such that $V(t)\le C t^{3/2}$.
\item Let $w>0$, $x_n\in M$ and $t_n\to +\infty$. If $x_n$ is in the $w$-thick part of $(M,t_n^{-1}g(t_n))$ for every $n$, then the sequence of 
pointed manifolds $(M,t_n^{-1} g(t_n),x_n)$ subconverges smoothly to a complete finite volume pointed 'hyperbolic' $3$-manifold of sectional curvature $-1/4$.
\end{enumerate}
\end{enumerate}
\end{prop}

\begin{proof} Note that $\vol \tilde B(\tilde x,t,r) \ge \vol B(x,t,r)$. Properties (i), (ii) 
with the stronger assumption $\vol B(x,t,r) \ge wr^3$ correspond to Perelman \cite[6.8, 7.3]{Per2}). For the extension to the universal cover see \cite[propositions 4.1, 4.2]{Bam:longtimeI}. We remark that we extend the curvature controls to the full ball, 
as in \cite[Sec. 11.2.3]{B3MP} (cf. \cite[Remark 11.2.12]{B3MP}). Property (iii) follows from Perelman \cite[7.1,7.2]{Per2}. For more details one can see Section 11.2 in \cite{B3MP}, using technical theorems 11.1.3 and 11.1.6. The assumption on the volume is used to prove that limits of rescaled parabolic neighbourhoods are hyperbolic (cf Proposition~11.2.3). 
\end{proof}

\begin{rem}
The hypothesis that $M$ is irreducible is not essential here, but since our Ricci flow with $(r(\cdot),\delta(\cdot))$-bubbling-off  
is defined for this situation, it makes sense to keep this assumption throughout.
\end{rem}

For later purposes, namely to prove that cuspidal tori in the appearing hyperbolic pieces are incompressible in $M$, we need the following 
improvement of Proposition~\ref{prop:thick thin}(iii)(b), which gives convergence of flows rather than metrics. With the notations of Proposition~\ref{prop:thick thin}, 
we define $g_{n} := t_{n}^{-1}g(t_{n})$ and $g_{n}(t):=t_{n}^{-1}g(tt_{n})$, the latter being a Ricci flow with bubling-off such that 
$g_{n}(1)=g_{n}$. If $g_{\hyp}$ denotes the `hyperbolic' metric of sectional curvature $-1/4$, then the Ricci flow $g_{\hyp}(t)$ 
satisfying $g_{hyp}(1)=g_{\hyp}$ is simply $g_{\hyp}(t)=tg_{\hyp}$. Consider $w>0$, $t_{n} \to \infty$ and  $x_{n}$ in the $w$-thick 
part of $(M,g_{n})$.  By Proposition~\ref{prop:thick thin} there exists a (sub)-sequence of $(M,g_{n},x_{n})$ converging  smoothly to  
$(H,g_{\hyp}, x_{\infty})$.  By relabeling, we can assume that the sequence converges. Then we have:
\begin{prop}\label{prop:thick thin persist} The sequence $(M\times[1,2], g_{n}(t),(x_{n},1))$ converges smoothly to 
$(H \times[1,2], g_{\hyp}(t),(x_{\infty},1))$.
\end{prop}

\begin{proof} 
We need to show that, for all $A>0$, for all $n$ large enough, the rescaled parabolic ball $B(\bar x_{n},1,A)\times [1,2]$ is $A^{-1}$-close 
to $B(x_{\infty},1,A) \times [1,2]$. In what follows we put a bar on $x_n$ to indicate that the ball 
is w.r.t $g_n(t)$.\\

We use the Persistence Theorem \cite[Theorem 8.1.3]{B3MP}, the hyperbolic limit $(H \times[1,2], g_{\hyp}(t),(x_{\infty},1))$ being the model $\calm_{0}$ in the sense of \cite[page 89]{B3MP}. Fix $A>1$ and let $\rho:=\rho(\calm_{0},A,1)\geq A$ be the parameter from the Persistence Theorem. By definition 
of $(H,g_{\hyp}, x_{\infty})$, note that $(B(\bar x_{n},1,\rho),g_{n})$ is $\rho^{-1}$-close to $(B(x_{\infty},1,\rho),g_{\hyp})$ for all sufficiently large $n$, satisfying 
assumption (ii) of \cite[Theorem 8.1.3]{B3MP}. To verify the other assumptions, we adapt arguments of \cite[Lemma 88.1]{Kle-Lot:topology} 
to our situation. In particular we have to take care of hyperbolic pieces appearing in a large $3$-ball affected by a metric surgery.
This is ruled out by a volume argument.\\

So we consider for each $n$, $T_{n} \in [t_n,2t_n]$ maximal such that 
\begin{enumerate}[(i)]
\item $B(x_{n},t_{n},\rho\sqrt{t_{n}}) \times [t_{n},T_{n}]$ is unscathed,
\item $|2t\Ric(x,t) + g(x,t)|_{g(t)} \leqslant 10^{-6}$ there.
\end{enumerate}

The case $T_{n}=t_n$, where $t_{n}$ is a singular time and a surgery affects the ball just at that time, is not \emph{a priori} excluded.
Note that (ii) implies $|\Rm_{g_{n}}| \leqslant 1$ on the considered neighbourhood: one has 
$\Ric_{g(t)} \approx -\frac{1}{2t}g(t)$ for $t \in [t_n,T_n]$, or $\Ric_{g(tt_{n})} \approx -\frac{1}{2tt_{n}}g(tt_{n})$ for $t\in[1,T_n/t_n]$, and then 
$\Ric_{g_{n}(t)} = \Ric_{t_{n}^{-1}g(tt_{n})} \approx -\frac{1}{2tt_{n}}g(tt_{n}) = -\frac{1}{2t}g_{n}(t)$. Thus the sectional 
curvatures of $g_{n}(t)$ remain in $[-\frac{1}{4}-\frac{1}{100},-\frac{1}{8}+\frac{1}{100}]$ for $A$ large enough.

We let $\bar T_n:=T_n/t_n \in [1,2]$ denote the rescaled final time. The assumptions of \cite[Theorem 8.1.3]{B3MP} being satisfied on $B(\bar x_{n},1,\rho)\times [1,\bar T_n]$, the conclusion holds on 
 $B(\bar x_{n},1,A)\times [1,\bar T_n]$, that is $(B(\bar x_{n},1,A)\times [1,\bar T_n],g_n(t))$ is $A^{-1}$-close to 
$(B(x_{\infty},1,A) \times [1,\bar T_n],g_{\hyp}(t))$.
 
\begin{claim}\label{claim:two} For all $n$ large enough, $\bar T_n = 2$. 
\end{claim}
\begin{proof}[Proof of Claim~\ref{claim:two}]

 We first prove that there are at most finitely many integers $n$ such that $T_{n}$ 
is a singular time where $B(x_{n},t_{n},\rho\sqrt{t_{n}})$ is scathed, that 
is $g_{+}(x,T_{n})\not=g(x,T_{n})$ for 
some $x\in B(x_{n},t_{n},\rho\sqrt{t_{n}})$.  

We first describe the idea of the proof. Assume that $T_{n}$ is such a singular time. By 
definition of our $(r,\delta)$-surgery, there is a surgery $3$-ball $B\ni x$ 
whose boundary $\partial B$ is the middle sphere of a strong $\delta$-neck 
with scalar curvature $\approx h^{-2}(T_{n})>>0$, where $h(T_{n})$ is the cutoff parameter 
at time $T_n$. By assumption (ii) above, $R<0$ at time $T_{n}$ on 
$B(x_{n},t_{n},\rho\sqrt{t_{n}})$, hence $\partial B \cap B(x_{n},t_{n},\rho\sqrt{t_{n}}) 
= \emptyset$. It follows that 
$B(x_{n},t_{n},\rho\sqrt{t_{n}}) \subset B$, which is an almost standard 
cap for $g_{+}(T_{n})$. For the pre-surgery metric, the persistence theorem implies 
that $(B(x_{n},t_{n},A\sqrt{t_{n}}),g(T_n))$ is almost homothetic to 
a (large) piece of the hyperbolic manifold $H$. Hence the surgery shrinks this piece to a 
small standard cap, decreasing volume by a definite amount. As moreover $t^{-1}g(t)$ is 
volume decreasing along time, volume would become negative if there were too many such singular times, yielding a contradiction. We now 
go into the details.
 
 Let $\mu >0$ be the volume of the unit ball in $(H,g_\hyp(1))$ centred at $x_\infty$, $B_\hyp:= B(x_\infty,1,1)$. 
 For any $t\geqslant 1$ we then have $\vol_{g_{\hyp}(t)}(B_\hyp)=t^{3/2}\vol_{g_{\hyp}}(B_\hyp) = t^{3/2}\mu$. We assume $A>1$, so that for $n$ large enough, by closeness at time $\bar T_n$ between $g_n(\cdot)$ and $g_\hyp(\cdot)$ we have: 
$$\vol_{g_{n}(\bar T_n)}(B(\bar x_{n},1,A)) \geqslant \frac{1}{2} \vol_{g_{\hyp}(\bar T_n)} (B_\hyp) = (\bar T_n)^{3/2}\frac{\mu}{2}.$$
Assume that $T_n$ is a singular time such that  $g_{+}(x,T_{n})\not=g(x,T_{n})$ for 
some $x\in B(x_{n},t_{n},\rho\sqrt{t_{n}})$ and let  $B\ni x$ be a surgery $3$-ball as 
discussed above.
As $B$ contains $B(\bar x_{n},1,\rho)$ and  
$\rho \geq A$, we also have  
$$\vol_{g_{n}(\bar T_n)} (B) > \bar T_n^{3/2}\frac{\mu}{2}.$$ 
For the unscaled metric $g(T_n)=t_ng_n(T_n/t_n)=t_ng_n(\bar T_n)$ we then have, before surgery, 
$\vol_{g(T_n)}(B) = t_n^{3/2} \vol_{g_{n}(\bar T_n)}(B) \geq (t_n\bar T_n)^{3/2}\frac{\mu}{2}=T_n^{3/2}\frac{\mu}{2}$. 
After surgery, $\vol_{g_{+}(T_{n})}(B)$ is comparable to $h^{3}(T_{n})$. Computing the difference of volumes gives: 
\begin{eqnarray*}
 \vol_{g_{+}(T_{n})}(B) - \vol_{g(T_{n})}(B) & \leq & c.h^{3}(t'_{n}) -  T_n^{3/2}\frac{\mu}{2} < -T_n^{3/2} \frac{\mu}{4}.
\end{eqnarray*}
for all $n$ large enough. 
Since $g_{+}(t) \leqslant g(t)$ on the whole manifold, we have 
\begin{equation}
 \vol_{g_{+}(T_{n})}(M) - \vol_{g(T_{n})}(M)  <  -T_n^{3/2} \frac{\mu}{4}, \label{eq:volume}
\end{equation}
for all $n$ large enough. 
Now the proof of \cite[Proposition 11.2.1]{B3MP} shows that $(t+\frac{1}{4})^{-1}g(t)$ is volume non-increasing along a smooth Ricci flow. 
Since  $g_{+}\leqslant g$ at singular times, this monotonicity holds for a Ricci flow with bubbling-off. One easily deduces by comparing 
the $(t+1/4)^{-1}$ and the $t^{-1}$ scaling that $t^{-1}g(t)$ is also volume decreasing. Precisely, let us now set 
$\bar g(t):=t^{-1}g(t)$, then for all $t' > t$: 
$$  \vol_{\bar g(t')}(M) \leqslant \left(\frac{t'+1/4}{t'}  \frac{t}{t+1/4}   \right)^{3/2} \vol_{\bar g(t)}(M)<\vol_{\bar g(t)}(M).$$
It particular, the sequence $\vol_{\bar g(t_{n})} (M)$ is decreasing. Moreover, if $[t_{n},t_{m}]$ contains a singular time $T_{n}$ as above, then using \eqref{eq:volume} in the second inequality, we get:
$$  \vol_{\bar g(t_m)}  (M)   \leqslant  \vol_{\bar g_{+}(T_{n})}(M) 
  <   \vol_{\bar g(T_{n})}   (M) -\frac{\mu}{4} 
 \leqslant   \vol_{\bar g(t_{n})}  (M) -\frac{\mu}{4}. 
$$
On the other hand, $\vol_{\bar g(t_n)}(M) >0$. Thus there are at most finitely many such singular times. We conclude that 
$B(x_{n},t_{n},\rho\sqrt{t_{n}})$ is unscathed at time $T_{n}$ for all $n$ large enough.\\

From now on we suppose $n$ large enough such that $B(x_{n},t_{n},\rho\sqrt{t_{n}}) \times [t_n,T_n]$ is unscathed. Recall that singular times form a discrete subset of $\RR$, hence there exists $\sigma_{n} >0$ such that $B(x_{n},t_{n},\rho\sqrt{t_{n}})$ is unscathed on $[t_{n},T_n+\sigma_{n}]$. By maximality of $\bar T_n$, when 
$\bar T_n < 2$ we must have $|2t\Ric(x,t) + g(x,t)|_{g(t)} = 10^{-6}$ at time $T_n$ for some 
$x\in\overline{B(x_{n},t_{n},\rho\sqrt{t_{n}})}$. Otherwise 
by continuity we find $\sigma_{n}$ small enough such that (ii) holds on $[t_{n},T_n+\sigma_{n}]\subset [t_{n},2t_{n}]$, contradicting the maximality of $\bar T_n$.\\

We now show that for all large $n$, $|2t\Ric(x,t) + g(x,t)|_{g(t)} < 10^{-6}$ at time $T_{n}$ on $\overline{B(x_{n},t_{n},\rho\sqrt{t_{n}})}$, which will imply that $\bar T_n =2$ by the discussion above. Using the $A^{-1}$-closeness of the rescaled parabolic ball $B(\bar x_{n},1,A) \times [1,\bar T_n]$ with $B(\bar x_{\infty},1,A) \times [1,\bar T_n]$, one can check 
 that $x_n$ is in the $w'$-thick part of $(M,{T_n}^{-1}g(T_n))$, for some fixed $w'>0$, for all $n$ large enough. Proposition 
\ref{prop:thick thin}(b) then implies that ${T_n}^{-1}g(T_n)$ becomes arbitrarily close to being hyperbolic on any fixed 
ball (w.r.t ${T_n}^{-1}g(T_n)$) centred at $x_n$, when $n \to \infty$. Controlling the distortion of distances on 
$B(x_n,t_n,\rho\sqrt{t_{n}})\times [t_n, T_n]$ (with the estimates (ii)), one can conclude that  
$|2t\Ric(x,t) + g(x,t)|_{g(t)} < 10^{-6}$ on $\overline{B(x_n,t_n,\rho\sqrt{t_{n}})}$ at time $T_n$ for $n$ large enough. The details are left to the reader.  Together with the first part of the proof and the maximality of $\bar T_{n}$, this implies that $\bar T_{n} = 2$ for $n$ large enough, proving  Claim~\ref{claim:two}.
\end{proof}

As already noted, we then have, by the Persistence Theorem, that $B(x_{n},1,A)\times [1,2]$, with the rescaled flow $g_{n}(t)$, is $A^{-1}$-close to $B(x_{\infty},1,A) \times [1,2]$ for all $n$ large enough. This concludes the proof of Proposition~\ref{prop:thick thin persist}.
\end{proof}

From Proposition \ref{prop:thick thin persist} one easily obtains:

\begin{corol}\label{cor:persist} Given $w >0$ there exist a number $T=T(w)>0$ and a nonincreasing function $\beta = \beta_w : [T,+\infty) \to (0,+\infty)$ tending to $0$ at $+\infty$ such that if $(x,t)$ is in the $w$-thick part of $(M,t^{-1}g(t))$ with $t \ge T$, then there exists 
a pointed hyperbolic manifold $(H,g_{\hyp},\ast)$ such that:
\begin{enumerate}[(i)]
\item $P(x,t,\beta(t)^{-1}\sqrt{t},t)$ is $\beta(t)$-homothetic to $P(\ast,1,\beta(t)^{-1},1) \subset H\times [1,2]$, endowed with 
$g_{\hyp}(s)=sg_{\hyp}(1))$,
\item For all $y \in B(x,t,\beta(t)^{-1}\sqrt{t})$ and $s \in [t,2t]$, 
$$ \lVert \bar g(y,s) - \bar g(y,t) \rVert < \beta,$$
where the norm is in the $C^{[\beta^{-1}]}$-topology w.r.t the metric $\bar g(t)=t^{-1}g(t)$.
\end{enumerate}
\end{corol}

%

\section{Incompressibility of the boundary tori}\label{sec:incompressible}
We prove that under the hypotheses of the previous section the tori that separate the thick part from the thin part are incompressible.

More precisely, we consider $M$ nonspherical, irreducible, not homeomorphic to $\RR^3$, endowed with a complete finite volume Ricci flow with 
$(r(\cdot),\delta(\cdot))$-bubbling-off $g(\cdot)$ such that $\delta(\cdot)\le \bar\delta(\cdot)$, \emph{and whose universal cover has bounded geometry} (for each time slice).  
We call \bydef{hyperbolic limit} a pointed `hyperbolic' manifold of finite volume and sectional curvature $-1/4$ that appears as a pointed limit of $(M,t_n^{-1} g(t_n),x_n)$ for 
some sequence $t_n\to\infty$. {\bf In this section, we assume the existence of at least one hyperbolic limit $(H,g_\hyp,*)$, 
which is supposed not to be closed.}\\

Given a hyperbolic limit $H$, we call \bydef{compact core of $H$},  a compact submanifold $\bar H \subset H$ whose complement consists of finitely many product neighbourhoods of the cusps. Then for large $n$, we have an approximating embedding $f_n:\bar H\to M$ which is almost isometric with respect to the metrics $g_\hyp$ and $t_n^{-1} g(t_n)$. The goal of this section is to prove the following result:

\begin{prop}\label{prop:incompressible}	
If $n$ is large enough, then for each component $T$ of $\bord \bar H$,  the image $f_n(T)$ is incompressible in $M$.
\end{prop}

We argue following Hamilton's paper \cite{hamilton:nonsingular}. 
A key tool is the stability of the hyperbolic limit $H$: it is a limit along the flow, not just along a sequence of times. We give 
a statement following Kleiner-Lott (cf.~\cite[Proposition 90.1]{Kle-Lot:topology}.)
\begin{prop}[Stability of thick part]\label{prop:stability}
There exist a number $T_0>0$, a nonincreasing function $\alpha:[T_0,+\infty) \to (0,+\infty)$ tending to $0$ at $+\infty$, 
a finite collection $\{(H_{1},\ast_{1}),\ldots,(H_{k},\ast_{k})\}$ of hyperbolic limits 
and a smooth family of smooth maps 
$$f(t) : B_{t} =  \bigcup_{i=1}^{k}B(*_{i},\alpha(t)^{-1}) \to M$$
 defined for $t\in [T_0,+\infty)$, such that
\begin{enumerate}[(i)]
\item The $C^{[\alpha(t)^{-1}]}$-norm of $t^{-1} f(t)^* g(t) - g_\hyp$ is less than $\alpha(t)$;
\item For every $t_0\ge T_0$ and every $x_0\in B_{t_{0}}$, the time-derivative at $t_0$ of the 
function $t\mapsto f(t)(x_0)$ is less than $\alpha(t_{0}) t_{0}^{-1/2}$.
\item $f(t)$ parametrises more and more of the thick part: the $\alpha(t)$-thick part of $(M,t^{-1}g(t))$ is contained in $\im(f(t))$.
\end{enumerate}
\end{prop}

The proof of \cite{Kle-Lot:topology} transfers directly to our situation, using Corollary \ref{cor:persist}. 
\begin{rem}\label{rem:stability} Any hyperbolic limit $H$ is isometric to one of the $H_{i}$. Indeed, let 
  $\ast \in H$ and $w>0$ be such that $\ast \in H^+(w)$. Then $x_n$ is in the $w/2$-thick part of $(M,t_{n}^{-1}g(t_{n}))$ for $n$ large enough. Assume that $f(t_{n})^{-1}(x_{n}) \in B(*_{i},\alpha(t_{n})^{-1})$ for a subsequence. Then $f(t_{n})^{-1}(x_{n})$ remains at bounded distance of $\ast_{i}$, otherwise it would go into a cusp contradicting the $w/2$-thickness of $x_{n}$. It follows that $(M,x_{n})$ and $(M,f(t_{n})(\ast_{i}))$ will have the same limit, up to an isometry.
 \end{rem}

\subsection{Proof of Proposition \ref{prop:incompressible}}

The proof of Hamilton \cite{hamilton:nonsingular} is by contradiction. Assuming that some torus is compressible, one finds an 
embedded compressing disk for each time further. Using Meeks and Yau \cite{Mee-Yau,Mee-Yau:smith}, the compressing disks can be chosen of least area. 
By controlling the rate of change of area of these disks, Hamilton shows that the area must go to zero in finite time---a contradiction.

Due to the possible noncompactness of our manifold, the existence of the least area compressing disks is not ensured: an area minimising sequence of disks can go deeper and deeper in an almost hyperbolic cusp. We will tackle this difficulty by considering the universal cover, which has bounded geometry (cf.	 Lemma \ref{lem:bd} and Addendum \ref{add:existence}), when necessary. 

Let us fix some notation. For all small $a>0$ we denote by $\bar H_a$ the compact core in $H$ whose boundary consists of 
horospherical tori of diameter $a$. By Proposition \ref{prop:stability} and Remark \ref{rem:stability}, 
we can assume that the  map $f(t)$ is 
defined on $B(\ast, \alpha(t)^{-1}) \supset \bar H_a$ for $t$ larger than some $T_a>0$. 
For all $t\geqslant T_a$ the image $f(t)(\bar H_a)$ is well defined and the compressibility in $M$ of a given boundary torus 
 $f(t)(\partial \bar H_a)$ does not depend on $t$ or $a$. We assume that some torus $\mathbf T$ of $\partial \bar H_a$ has 
compressible image in $M$. Below we refine the choice of the torus $\mathbf T$.

We define, for some fixed $a>0$,  
$$Y_{t} := f(t)(\bar H_a), \quad \mathbf{T}_{t}:=f(t)(\mathbf{T}) \quad \textrm{and} \quad W_{t} := M - \Int(Y_{t}).$$
 Our first task is to find a torus in $\partial Y_t$ which is compressible in $W_t$. Note that 
$\mathbf{T}_t$ is compressible in $M$, incompressible in $Y_t$ which is the core of a hyperbolic $3$-manifold, but not necessarily compressible in 
$W_t$: for example $Y_t$ could be contained in a solid torus and $\mathbf{T}_t$ compressible on this side. 

Consider the surface $\partial Y_t \subset M$ (not necessarily connected). As the induced map 
$\pi_1(\mathbf{T}_t) \to \pi_1(M)$, with base point choosen in $\mathbf{T}_t$, is noninjective by assumption, 
 Corollary~3.3 of Hatcher~\cite{Hat} tells that there is a compressing disk $D\subset M$, 
 with $\partial D \subset \partial Y_t$ homotopically non trivial and 
 $\Int(D) \subset M-\partial Y_t$. As $\Int(D)$ is not contained in $Y_t$, one has $\Int(D)\subset W_t$. Rename $\mathbf{T}_t$ the connected component of $\partial Y_t$ which contains $\partial D$ and $\mathbf{T}\subset \partial \bar H_a$ its 
 $f(t)$-preimage. Then $\mathbf{T}_t$ is compressible in $W_t$. 

Let $X_t$ be the connected component of $W_t$ which contains $D$. Using \cite[Lemma A.3.1]{B3MP} we have two exclusive possibilities:
\begin{enumerate}[(i)]
 \item $X_t$ is a solid torus. It has convex boundary, hence Meeks-Yau \cite[Theorem 3]{Mee-Yau} provide a least area compressing disk $D^2_t \subset X_t$ where $\partial D^2 \subset \mathbf{T}_t$ is in a given nontrivial free homotopy class.
 \item $\mathbf{T}_t$ does not bound a solid torus and $Y_t$ is contained in a $3$-ball $B$. Then $Y_t$ lifts isometrically to a $3$-ball in the universal cover 
$(\tilde M,\tilde g(t))$. Let $\tilde Y_t$ be a copy of $Y_t$ in $\tilde M$. By \cite{Hat} again, there is a torus 
$\tilde{\mathbf{T}_t }\subset \partial \tilde Y_t$ compressible in $\tilde M - \partial \tilde Y_t$, hence in 
$\tilde M - \tilde Y_t$. We denote by $\tilde X_t$ the connected component of $\tilde M - \Int(\tilde Y_t)$ in which $\tilde{\mathbf{T}_t}$ is compressible. 
As $(\tilde M,\tilde g(t))$ has bounded geometry, by \cite[Theorem 1]{Mee-Yau:smith} there is a compressing disk $D^2_t \subset \tilde X_t$ of least area with $\partial D^2_t \subset \tilde{\mathbf{T}_t}$ in a given nontrivial free homotopy class.

\end{enumerate}

We define a function $A : [T_a,+\infty)  \to (0,+\infty)$ by letting $A(t)$ be the infimum of the areas of such embedded disks. Similarly to  
\cite[Lemma 91.12]{Kle-Lot:topology} we have

\begin{lem}\label{lem:rate} For every $D>0$, there is a number $a_0>0$ with the following property. 
Given $a \in (0,a_0)$ there exists $T'_a>0$ such that for all $t_0\geqslant T'_a$ there is a piecewise smooth function $\bar A$ defined in a 
neighbourhood of $t_0$ such that $\bar A(t_0)=A(t_0)$, $\bar A \geqslant A$ everywhere, and 
$$ \bar A'(t_0) < \frac{3}{4}\left(\frac{1}{t_0+\frac{1}{4}} \right) A(t_0) -2\pi + D$$
if $\bar A$ is smooth at $t_0$, and $\lim_{t \to t_0,t>t_0} \bar A(t) \leqslant \bar A(t_0)$ if not.
 \end{lem}

\begin{proof}
The proof is similar to the proof of \cite[Lemma 91.12]{Kle-Lot:topology}, and somewhat simpler as we don't have topological surgeries. Recall 
that our Ricci flow with bubbling-off $g(t)$ is non increasing at singular times, hence the unscathedness of least area compressing disks 
(\cite[Lemma 91.10]{Kle-Lot:topology}) 
is not needed: we have $\lim_{t \to t_0,t>t_0} A(t) \leqslant A(t_0)$ if $t_0$ is singular. However, something must be said about 
\cite[Lemma 91.11]{Kle-Lot:topology}. This lemma asserts that given $D>0$, there is $a_0>0$ such that 
for $a \in (0,a_0)$ and $\mathbf{T} \subset H$ a horospherical torus of diameter $a$, for all $t$ large enough 
$\int_{\partial D^2_t} \kappa_{\partial D^2_t} ds \leqslant \frac{D}{2}$ and $\length(\partial D^2_t) \leqslant \frac{D}{2}\sqrt{t}$, where 
$\kappa_{\partial D^2_t}$ is the geodesic curvature of $\partial D^2_t$. Its proof relies on the fact that an arbitrarily large 
collar neighbourhood of $\mathbf{T}_t$ in $W_t$ is close (for the rescaled metric $t^{-1}g(t)$) to a hyperbolic cusp if $t$ is large enough. In case (1) above, 
this holds on $X_t \cap f(t)B(\ast,\alpha(t)^{-1}))$ by Proposition \ref{prop:stability}. In case (2) observe 
that $f(t)(B(\ast,\alpha(t)^{-1}))$ is homotopically equivalent to the 
compact core $\bar H_t$, hence lifts isometrically to $(\tilde M,\tilde g(t))$. It follows that $\tilde X_t$ also has an arbitrarily 
large collar neighbourhood of $\tilde{\mathbf{T}}_t$ close to a hyperbolic cusp. 

The rest of the proof is identical to the proof of \cite[Lemma 91.12]{Kle-Lot:topology} and hence omitted.
\end{proof}
In particular $A$ is upper semi-continous from the right. Note also that as $A$ is defined as an infimum and $g(t_k)$ and $g(t)$ are $(1+\epsi_k)$-bilischitz when times $t_k \nearrow t$, for some $\epsi_k \to 0$, $A$ is lower semi-continuous from the left. 

Fix $D<2\pi$, $a \in (0,a_0)$ and $T'_a$ as in Lemma \ref{lem:rate}. 
Then consider the solution  $\hat A : [T'_a,+\infty) \to \RR$ of the ODE 
$$ \hat A' = \frac{3}{4}\left(\frac{1}{t+\frac{1}{4}} \right) \hat A -2\pi + D$$
 with initial condition $\hat A(T'_a)=A(T'_a)$. By a continuity argument, $A(t) \leqslant \hat A(t)$ for all 
 $t \geqslant T'_a$. However, from the ODE we have
$$\hat A(t) \left(t+\frac{1}{4} \right)^{-3/4} =4(-2\pi +D) \left(t+\frac{1}{4}\right)^{1/4} + \const,$$
which implies that $\hat A(t) <0$ for large $t$, contradicting the fact that $A(t)>0$.

This finishes the proof of Proposition \ref{prop:stability}.


\section{A Collapsing Theorem}
In this section we state a version of the collapsing theorem \cite[Theorem 0.2]{Mor-Tia:collapsing} in the context of 
manifolds with cusp-like metrics.

Let $(M_n,g_n)$ be a sequence of Riemannian $3$-manifolds.

\begin{defi}
We say that $g_n$ has \bydef{locally controlled curvature in the sense of Perelman} if for all 
$w>0$ there exist $\bar r(w)>0$ and $K(w) >0$ such that for $n$ large enough , if $0<r \le \bar r(w)$, if
$x \in (M_n,g_n)$ satisfies $\vol B(x,r) \ge w r^3$ and $\sec \ge -r^{-2}$ on $B(x,r)$ then 
$|\Rm(x)| \leqslant Kr^{-2}$, $|\nabla \Rm(x)| \leqslant Kr^{-3}$ and 
$|\nabla^2 \Rm(x) | \leqslant Kr^{-4}$ on $B(x,r)$.
\end{defi}

\begin{rem}\label{rem:curvature locally}
Note that if $g_n = {t_n}^{-1}g(t_n)$, where $g(\cdot)$ is as in
Proposition~\ref{thm:thick thin} and $t_n \to \infty$, then $g_n$ has locally controlled curvature in the sense of Perelman.
\end{rem}

\begin{defi}
 We say that $(g_n)$ \bydef{collapses} if there exists a sequence $w_n \to 0$ of positive numbers such that 
$(M_n,g_n)$ is $w_n$-thin for all $n$.
\end{defi}

From \cite[Theorem 0.2]{Mor-Tia:collapsing} we obtain:
\begin{theo}\label{thm:collapsing} Assume that $(M_n,g_n)$ is a sequence of complete Riemannian oriented $3$-manifolds
 such that 
\begin{enumerate}[\indent (i)]
 \item $g_n$ is a cusp-like metric for each $n$,
 \item $(g_n)$ collapses, 
\item $(g_n)$ has locally controlled curvature in the sense of Perelman,
\end{enumerate}
then for all $n$ large enough $M_n$ is a graph manifold.
 \end{theo}
The manifolds in \cite[Theorem 0.2]{Mor-Tia:collapsing} are assumed to be compact and may have convex boundary. Our cusp-like assumption (i) allows to apply their result 
by the following argument. First we deform each $g_n$ so that the sectional curvature is $-\frac{1}{4}$ on some neighbourhood of the ends, assumptions (ii),(iii) remaining 
true. Let $w_n \to 0$ be a sequence of positive numbers such that $g_n$ is $w_n$-thin. 
For each $n$, we can take a neighbourhood $U_n$ of the ends of $M_n$, with horospherical boundary, small enough so that the complement $M'_n=M_n \setminus \Int U_n$ satisfies assumptions 
 of  \cite[Theorem 0.2]{Mor-Tia:collapsing} with collapsing numbers $w_n$, except for the convexity of the added boundary. Then we deform the metric on $M'_n$ 
 near the boundary into a reversed hyperbolic cusp so that the boundary becomes convex. It follows that $M'_n$, hence $M_n$, is a graph manifold for all $n$ large enough. In fact it should be clear from Morgan-Tian's proof that the convexity assumption is not necessary in this situation 
 (see the more general \cite[Proposition 5.1]{bam:longtimeII}).

\section{Proof of the main theorem}

Here we prove Theorem~\ref{thm:geometrisation}. We sketch the organisation 
of the proof. Let $(M,g_0)$ be a Riemannian $3$-manifold satisfying the hypotheses of this theorem. We also assume that $M$ is not a solid torus, 
is nonspherical and does not have a metric with $\Rm\ge 0$, otherwise it would be Seifert 
fibred and conclusion of Theorem~\ref{thm:geometrisation} holds.
We first define on $M$ a Ricci flow  with $(r(\cdot),\delta(\cdot))$-bubbling-off $g(\cdot)$, issued from $g_0$ and 
defined on $[0,+\infty)$. As mentioned before, we may have to pass to the universal cover. 
 By existence Theorem \ref{thm:existence} $g(\cdot)$ exists 
on a maximal interval $[0,\Tmax)$. The case $\Tmax < +\infty$ is ruled out using the fact that $(M,g(\Tmax))$ is covered by canonical neighbourhoods (see claim \ref{claim:alltime} below).
 Proposition~\ref{thm:thick thin} then provides a sequence $t_n \nearrow +\infty$ and connected open 
 subsets $H_{n} \subset M_n=(M,{t_n}^{-1}g(t_n))$, diffeomorphic to a complete, finite volume hyperbolic manifold $H$ (possibly empty). 
 We set $G_n := M_n \setminus H_n$. Proposition \ref{prop:incompressible} proves that the tori of $\partial \overline H_{n}$ (if $H \not=\emptyset$) are incompressible in $M$ for large $n$. In this case, the atoroidality assumption on $M$ implies that $H_n$ is	 diffeomorphic to $M$ and that each component of $G_n$ is a cuspidal end $T^2 \times [0,\infty)$ of $M_n$. 
Then $g(t)$ converges (in the pointed topology) to a complete, finite volume hyperbolic metric on $M$. In both 
cases ($H = \emptyset$ or $H \neq\emptyset$), $G_{n}$ collapses with curvature locally controlled in the sense 
of Perelman. If $H = \emptyset$, we conclude by collapsing theorem \ref{thm:collapsing} that $M_n=G_n$ is a 
graph manifold (hence Seifert fibred) for all $n$ large enough. If $H_n \neq \emptyset$, Proposition 4.2 
gives a continuous decomposition $M=H_t \cup G_t$ where $H_t$ is diffeomorphic to $M$, $g(t)$ is smooth and $|\Rm| \leq Ct^{-1}$ 
there, and $G_t$ is $\alpha(t)$-thin. We then use the topological/geometric description of the thin part presented in \cite{Bam:longtimeI, bam:longtimeII} 
to obtain that $|\Rm| \le Ct^{-1}$ on $G_t$, by the same argument as in \cite[Theorem 1.1]{Bam:longtimeI}.\\

%


\subsection{Setting up the proof}

Let $(\tilde M,\tilde g_0)$ be the Riemannian universal cover of $(M,g_0)$. By Lemma \ref{lem:bd} it has bounded geometry.  
Without loss of generality, we assume that it is normalised. If $M$ is compact, we can even assume that $g_0$ itself is normalised.

We now define a Riemannian $3$-manifold $(\bar M,\bar g_0)$ by setting $(\bar M,\bar g_0):=(M,g_0)$ if $M$ is compact, 
and $(\bar M,\bar g_0):=(\tilde M,\tilde g_0)$ otherwise. In either case, $\bar g_0$ is complete and 
normalised. By~\cite{msy}, $\bar M$ is irreducible. If $\bar M$ is spherical, then $M$ is spherical, contrary to the assumption. 
Henceforth, we assume that $\bar M$ is nonspherical. 

Thus Theorem~\ref{thm:existence} applies to $(\bar M,\bar g_0)$, where $\bar \delta(\cdot)$ is chosen from Theorem 
\ref{thm:thick thin}. Let $\bar g(\cdot)$ be a Ricci flow with bubbling-off on $\bar M$ with initial condition $\bar g_0$. By Addendum~\ref{add:existence}, we also have a Ricci flow with bubbling-off $g(\cdot)$ on $M$ with initial condition $g_0$ covered by $\bar g(\cdot)$.

\begin{claim}\label{claim:alltime}
The evolving metrics $g(\cdot)$ and $\bar g(\cdot)$ are defined on $[0,+\infty)$.
\end{claim}

\begin{proof}
If this is not true, then they are only defined up to some finite time $T$, and every point of $(\bar M,\bar g(T))$ is centre of an $\epsi_0$-neck or an $\epsi_0$-cap.
By Theorem~7.4 of~\cite{bbm:openflow}, $\bar M$ is diffeomorphic to $S^3$, $S^2\times S^1$, $S^2\times \Rr$ or $\Rr^3$.\footnote{This list is shorter than the corresponding list in~\cite{bbm:openflow} since we do not consider caps diffeomorphic to the punctured $RP^3$.} Since $\bar M$ is irreducible and nonspherical, $\bar M$ is diffeomorphic to $\Rr^3$. The complement of the neck-like part (cf.~again~\cite{dl:equi}) is a $3$-ball, which must be invariant by the action of the deck transformation group. Since this group acts freely, it is trivial. Thus $M=\bar M$.

Being covered by $\bar g(\cdot)$, the evolving metric $g(\cdot)$ is complete and of bounded sectional curvature. Hence by Remark~\ref{rem:finite volume}, $(M,g(T))$ has finite volume. By contrast, $(\bar M,\bar g(T))$ contains an infinite collection of pairwise disjoint $\epsi_0$-necks of controlled size, hence has infinite volume. This contradiction completes the proof of Claim~\ref{claim:alltime}.
\end{proof}

It follows from Claim~\ref{claim:alltime} that $\bar M$ carries an equivariant Ricci flow with bubbling-off $\bar g(\cdot)$ defined 
on $[0,+\infty)$ with initial condition $\bar g_0$.  
We denote by $g(\cdot)$ the quotient evolving metric on $M$. By Addendum~\ref{add:existence}, it is also a Ricci flow with $(r(\cdot),\delta(\cdot))$-bubbling-off. By Theorem \ref{thm:cusp-like}, $g(\cdot)$ remains cusp-like at infinity for all time. Now consider the alternative that follows 
the conclusion of Proposition \ref{thm:thick thin} part (iii) : Either 
\begin{enumerate}[(i)]
 \item there exist $w>0$, $t_n \to \infty$ such that the $w$-thick part of $(M,t_n^{-1}g(t_n))$ is nonempty for all $n$, or 
 \item there exist $w_n \to 0$, $t_n \to \infty$ such that the $w_n$-thick part of $(M,t_n^{-1}g(t_n))$ is empty for all $n$.
\end{enumerate}

We refer to the first case as the \emph{noncollapsing case} and to the second as the 
\emph{collapsing case}.\\

We denote by $g_n$ the metric ${t_n}^{-1}g(t_n)$. Note that $g_n$ has curvature locally controlled in the 
sense of Perelman (cf.~Remark~\ref{rem:curvature locally}).
We denote by $M_n$ the Riemannian manifold $(M,g_n)$, $M_n^+(w)$ its $w$-thick part, and 
$M_n^-(w)$ its $w$-thin part. In the collapsing case, $M_n=M_n^-(w_n)$ fits the assumptions of  Theorem \ref{thm:collapsing}. Hence it is a graph manifold for $n$ large enough.

%
%
%
%
%
Let us consider the other case.

\subsection{The noncollapsing case}

By assumption, there exist $w>0$ and a sequence $t_n \to \infty$ such that the $w$-thick part 
of $M_n$ is nonempty for all $n$.  Choose a sequence $x_n \in M_n^+(w)$. Up to extracting a subsequence, by part (iii) of 
Proposition \ref{thm:thick thin}, $(M_n,x_n)$ converges to a complete hyperbolic manifold  $(H,\ast)$ of finite volume. 
By definition of the convergence, there exist an exhaustion of 
$H$ by compact cores $\bar H_n \subset H$ and embeddings $f_n : (\bar H_n,\ast) \to (M,x_n)$ such that 
$|g_\hyp - f_{n}^\ast g_n|$ goes to zero.  Proposition~\ref{prop:stability} (stability of the thick part) 
gives $T_0>0$ and a nonincreasing function $\alpha : [T_{0},\infty) \to (0,\infty)$ tending to zero at infinity, and 
for $t\ge T_0$ embeddings $f(t) : B(\ast,\alpha(t)^{-1})\subset H \to M$ satisfying conclusions (i)--(iii) of this proposition. 
If $H$ is closed, the desired conclusion follows. From now on we assume that $H$ is not closed.
By Proposition \ref{prop:incompressible}, for each 
$m \in \NN$, for all $n$ large enough, each component of $f_n(\partial \bar H_m)$ is  an incompressible torus in $M$. 
Relabeling the $f_n$ we can assume that the property holds for $f_m(\partial \bar H_m)$ for all $m$. By atoroidality of $M$, it follows that 
$H_n := \Int f_n(\bar H_n) \subset M$ is diffeomorphic to $M$ for all $n$, and $G_n :=M \setminus H_n$ is a disjoint 
union of neighbourhoods of cuspidal ends of $M_n$.
For large $t\ge T_0$, choose a compact core $\bar H_t \subset B(\ast,\alpha(t)^{-1})$ such that $\partial \bar H_t$ consists of horospherical tori 
whose diameter goes to zero as $t \to \infty$. We assume moreover that $t \to \bar H_t$ is smooth. Set $H_t := f(t)(\bar H_t)\subset M$ and 
$G_t := M \setminus H_t$.  Then $H_t$ is diffeomorphic to $M$, 
$t \mapsto g(t)$ is smooth there and $|\Rm| \leq Ct^{-1}$ by closeness with $H$. On the other hand,  $G_t$ is $w(t)$-thin for some $w(t) \to 0$ as $t\to \infty$. 
There remains to prove that $G_t$ satisfies $|\Rm| \leq Ct^{-1}$ also, which will imply its unscathedness.

Consider a connected component $\calc(t)$ of $G_t$. For all large $t$, $\partial \calc(t)$ is an incompressible torus in $M$ with 
a collar neighbourhood $\alpha(t)$-close, w.r.t $t^{-1}g(t)$, to a collar neighbourhood of a horospherical torus in $H$. 
On the other hand, $\calc(t)$ is diffeomorphic to $T^2 \times [0,\infty)$ and its 
 end has a cusp-like structure, hence curvature also bounded by $Ct^{-1}$. There remains to control what happens in the middle of $\calc(t)$. 
 
We apply the topological/geometric description of the thin part obtained in \cite[Proposition 5.1]{bam:longtimeII} to a 
compact subset $\calc'(t) \subset \calc(t)$ which we define as follows.   

By Theorem \ref{thm:stability2} there is an  embedding $f_{\cusp} : T^2 \times [0,+\infty)  \to M$ and a function $b : [0,+\infty) \to [0,\infty)$ such that 
$$| (4t)^{-1}f_{\cusp}^{\ast}g(t) - g_{\hyp}|_{T^2 \times [b(t),+\infty)} < w(t)$$
and $f_\cusp(T^2\times [b(t),+\infty)) \subset \calc(t)$ is a neighbourhood of its end. 
Here $g_{\hyp}$ denotes a hyperbolic metric $e^{-2s}g_{\eucl} + ds^2$ (with sectional curvature $-1$) on 
$T^2 \times [0,+\infty)$. The metric $4g_\hyp$ may differ from the one on $H$. 
We can assume $b(t) \to \infty$. We define $\calc_{\cusp}(t):=f_{\cusp}(T^2 \times [b(t)+2,+\infty))$ and  
$$ \calc'(t) := \calc(t) \setminus \Int \calc_{\cusp}(t).$$ 

Now we fix functions $\bar r$, $K$ given by Proposition~\ref{thm:thick thin}, $\mu_1>0$ given by \cite[Lemma 5.2]{bam:longtimeII}, 
$w_1=w_1(\mu_1,\bar r,K)>0$ given by \cite[Proposition 5.1]{bam:longtimeII}. 

The closed subset $\calc'(t)$ satisfies the assumptions of 
the latter proposition for $t \geq T_1$  large enough such that $w(t) < w_1$. 
We now follow the proof of \cite[Theorem 1.1 on p.~23]{Bam:longtimeI}.
Decompose $\calc'(t)$ into closed subsets $V_1,V_2,V'_2$ as given by the proposition. The two boundary components of 
$\calc'(t)$ have to bound components of $V_1$. Either $\calc'(t)=V_1$ or the boundary components 
of $\calc'(t)$ bound components $\calc_1,\calc_2$ of $V_1$, which are diffeomorphic to $T^2 \times I$ and there 
is a component $\calc_3$ of $V_2$ adjacent to $\calc_1$. We prove that only the first case occurs, for all $t$ large enough.

%


\begin{lem}\label{lem:vone}
 For all $t$ large enough, $\calc'(t)=V_1$.
 \end{lem}

Before proving this lemma, we explain how to conclude the proof of the theorem. First \cite[Lemma 5.2(ii)]{bam:longtimeII} applies to any $x \in V_{1}$, giving 
$w_1=w_1(\mu_1,\bar r,K)>0$ such that
$$\vol \tilde B(\tilde x,t,\rho_{\sqrt{t}}(x,t)) \ge w_1(\rho_{\sqrt{t}}(x,t))^3,$$
for any lift $\tilde x \in \tilde M$ of $x$.
Let $\bar \rho=\bar \rho(w_1)>0$ be given by Proposition \ref{prop:bar rho bis}. If $\rho_{\sqrt{t}}(x,t) < \rho(x,t)$  then $\rho(x,t) \geq \sqrt{t} > \bar \rho \sqrt{t}$. 
If not, $\rho(x,t) = \rho_{\sqrt{t}}(x,t)$ and Proposition \ref{prop:bar rho bis} (ii) implies  
$$\rho_{\sqrt t}(x,t) \geq \bar \rho \sqrt{t}$$ 
if $t$ is large enough (larger than $\bar T =\bar T(w_1)$). In both cases, $\rho_{\sqrt t} \geq \bar \rho \sqrt{t}$.
Then Proposition \ref{prop:bar rho bis}(i) with $r=\rho_{\sqrt{t}}(x,t)$ implies $|\Rm| \le C(w_1)t^{-1}$ at $(x,t)$ for some 
$C=C(w_1)>0$. 
Thus the proof of the theorem is finished if $\calc'(t) = V_{1}$ for all large $t$. \\

We now prove Lemma~\ref{lem:vone}, arguing by contradiction. Set $T_2 :=\max\{T_0,T_1,\bar T\}$.
 Assume that there exist arbitrarily large times $t \geq T_2$ such that $\calc'(t) \not=V_{1}$. At any of these times, the  
  $S^1$-fibres of $\calc_3$ are homotopic to a fibre of $\partial \calc_1$, by \cite[Proposition 5.1(b2)]{bam:longtimeII}. 
By incompressibility of $\partial \calc_{1}$ in $M$, this curve generates an infinite cyclic subgroup in $\pi_1(M)$. Then 
\cite[Lemma 5.2(i)]{bam:longtimeII} applies to any $x \in \calc_3 \cap V_{2,\reg}$ and gives 
$$ \vol \tilde B(\tilde x,t,\rho_{\sqrt{t}}(x,t)) \ge w_1(\rho_{\sqrt{t}}(x,t))^3,$$
for any lift $\tilde x$ of $x$, and 
hence $\rho_{\sqrt{t}}(x,t) \geq \bar \rho \sqrt{t}$ as above.
 Moreover 
 \cite[Proposition 5.1(c3)]{bam:longtimeII} gives $s=s_2(\mu_1,\bar r,K) \in (0,1/10)$,  an open set $U $ such that 
\begin{equation}
 B(x,t,\frac{1}{2}s\rho_{\sqrt{t}}(x,t)) \subset U \subset  B(x,t, s\rho_{\sqrt{t}}(x,t)), \label{eq:U}
\end{equation}
and a $2$-Lipschitz map $p : U \to \RR^2$ whose image contains $B(0,\frac{1}{4}s\rho_{\sqrt{t}} (x,t)) \subset \RR^2$ 
and whose fibres are homotopic to fibres of $\calc_3$, hence noncontractible in $M$.

Now consider any noncontractible loop $\gamma \subset \calc'(T_2)$. Define for all $t\geq T_2$, 
$\gamma_1(t) \subset \partial \calc(t)$ freely homotopic to $\gamma$ such that 
$f(t)^{-1} \circ \gamma_1(t)$ is geodesic in $\partial \bar H_t$ and evolves by parallel transport in $H$ w.r.t. $t$.
On the side of the cusp, define $\gamma_2(t) \subset \partial \calc_{\cusp}(t)$ freely homotopic to $\gamma$ such that 
$f_{\cusp}^{-1}\circ \gamma_2(t) \subset T^2 \times \{b(t)+2\}$ is geodesic in $(T^2,g_\eucl)$ and 
evolves by parallel transport (at speed $b'$).

In particular $\gamma_1(t) \subset \calc_1$ and $\gamma_2(t) \subset \calc_{2}$ at each time when these sets are defined (that is 
when $\calc'(t) \not=V_{1}$) and these loops are freely homotopic in $\calc'(t)$.
Let $A(t)$ be the infimum of the areas of all smooth homotopies $H : S^1 \times [0,1] \to \calc'(t)$ connecting 
$\gamma_1(t)$ to $\gamma_2(t)$.

\begin{claim}\label{claim:four}
 $t^{-1}A(t) \to 0$ as $t \to \infty$.
\end{claim}

\begin{proof}
It is identical to \cite[Lemma 8.2]{Bam:longtimeI}, except that we have to account for the fact that 
$\partial_{t} \gamma_{2}(t)$ may a priori not be bounded. This estimate appears when we compute the area added to the 
homotopy by moving the boundary curves. The infinitesimal added area 
to the homotopy due to the deplacement of $\gamma_1$ is negative (we can assume $\alpha' >0$), hence neglected. 
The contribution of $\gamma_2$, by closeness with the hyperbolic cusp, is bounded by $Ct.e^{-b}b'$. On the other hand, 
 the normalised length $t^{-1/2} \ell(\gamma_i) \to 0$ and the normalised geodesic curvature 
$t\kappa(\gamma_i(t)) < C$, by closeness with the hyperbolic situation.
Let us denote $L(t) = t^{-1/2}(\ell (\gamma_1(t))+\ell(\gamma_2(t))$. 
Computations in \cite[Lemma 8.2]{Bam:longtimeI} give 
(compare with equation (8.1) there)
\begin{equation*}
 \frac{d}{dt^+}(t^{-1}A(t)) \leq -\frac{A(t)}{4t^2}+ C\left(\frac{L(t)}{t} + e^{-b}b'\right).\label{eq:bamler}
  \end{equation*}
Denoting $y(t)=t^{-1}A(t)$ this gives the differential inequality 
$$ \frac{d}{dt^+}y \leq -y/4t + C(t^{-1}L+ e^{-b}b').$$ 
Using the standard method, one obtains that $y(t)=K(t)t^{-1/4}$ where 
$$\frac{d}{dt^+}K \leq Ct^{1/4}\left(t^{-1}L+  e^{-b}b' \right).$$ 
We can assume that $L(t)$ is almost nonincreasing, that is that for any $T_2 \le a \le t$, one has $L(t) \le 2L(a)$. Then 
for $T_2 \le a \le t$,
\begin{eqnarray*}
K(t) - K(a) & \leq & C\left (\int_a^t \left(u^{-3/4} L(u) + u^{1/4}e^{-b}b'\right)\, du\right)\\
 & \le & C\left(2L(a)\int_a^t u^{-3/4}\, du \, \, + t^{1/4} \int_a^t e^{-b}b'\, du\right)\\
    & \leq &  C\left( 8L(a)t^{1/4} + t^{1/4} e^{-b(a)} \right), 
\end{eqnarray*}
hence
\begin{eqnarray*}
y(t) & \leq & \frac{K(a)}{t^{1/4}} + \frac{K(t)-K(a)}{t^{1/4}} \\
  & \leq & \frac{K(a)}{t^{1/4}} + C\left(8L(a)+e^{-b(a)} \right),
\end{eqnarray*}
which is arbitrary small by taking $a$ then $t$ large enough.
\end{proof}

We conclude the proof of Lemma~\ref{lem:vone}. The argument is the same as the one given in 
\cite{Bam:longtimeI}. Consider smooth loops $\gamma,\beta$ in 
$\calc'(t)$ generating $\pi_1 \calc'(t)$. Let $\gamma_i(t)$, resp. $\beta_i(t)$, $i=1,2$, defined as above, 
freely homotopic to $\gamma$, resp. $\beta$. Let $A(t)$, resp. $B(t)$, be the infimum of the areas of all smooth 
homotopies  connecting $\gamma_1(t)$ to $\gamma_2(t)$, resp. $\beta_1(t)$ to $\beta_2(t)$. By Claim~\ref{claim:four}, 
\begin{equation} 
 t^{-1}A(t) + t^{-1} B(t) \to 0  \label{eq:area}
\end{equation}  
as $t \to \infty$. On the other hand let $H_\gamma$, resp. $H_\beta$, be any of these homotopies. At any time $t$ where $\calc_3$ is defined, 
any fibre of the projection $p : U \to \RR^2$ is a noncontractible loop $\subset \calc_3$, 
hence it intersects at least once the homotopies $H_\gamma$,$H_\beta$. For all such times $t$ large enough one has, 
using the fact that $p$ is $2$-bilipschitz and equation \eqref{eq:U}, that 
 $$\area(H_\gamma) + \area(H_\beta) \geq \frac{1}{4} \vol(p(U)) \geq cs^2\bar \rho t,$$
for some constant $c=c(s,\bar \rho)>0$. This contradicts \eqref{eq:area}.

\bibliographystyle{alpha}
\bibliography{ricci}

\end{document}

%% file: truncation.pstex_t
\begin{picture}(0,0)%
\includegraphics{truncation.pstex}%
\end{picture}%
\setlength{\unitlength}{2072sp}%
\begingroup\makeatletter\ifx\SetFigFont\undefined%
\gdef\SetFigFont#1#2#3#4#5{%
  \reset@font\fontsize{#1}{#2pt}%
  \fontfamily{#3}\fontseries{#4}\fontshape{#5}%
  \selectfont}%
\fi\endgroup%
\begin{picture}(11876,5176)(616,-5681)
\put(10486,-4831){\makebox(0,0)[lb]{\smash{{\SetFigFont{9}{10.8}{\familydefault}{\mddefault}{\updefault}$V$}}}}
\put(6481,-4426){\makebox(0,0)[lb]{\smash{{\SetFigFont{9}{10.8}{\familydefault}{\mddefault}{\updefault}$W'$}}}}
\put(1981,-2356){\makebox(0,0)[lb]{\smash{{\SetFigFont{9}{10.8}{\familydefault}{\mddefault}{\updefault}$U$}}}}
\put(4771,-5281){\makebox(0,0)[lb]{\smash{{\SetFigFont{9}{10.8}{\familydefault}{\mddefault}{\updefault}$W$}}}}
\put(3691,-826){\makebox(0,0)[lb]{\smash{{\SetFigFont{9}{10.8}{\familydefault}{\mddefault}{\updefault}$0$}}}}
\put(631,-736){\makebox(0,0)[lb]{\smash{{\SetFigFont{9}{10.8}{\familydefault}{\mddefault}{\updefault}$-2\epsi^{-1}$}}}}
\put(6211,-736){\makebox(0,0)[lb]{\smash{{\SetFigFont{9}{10.8}{\familydefault}{\mddefault}{\updefault}$2\epsi^{-1}$}}}}
\put(2116,-3751){\makebox(0,0)[lb]{\smash{{\SetFigFont{9}{10.8}{\familydefault}{\mddefault}{\updefault}$-2\epsi^{-1}$}}}}
\put(7471,-3661){\makebox(0,0)[lb]{\smash{{\SetFigFont{9}{10.8}{\familydefault}{\mddefault}{\updefault}$2\epsi^{-1}$}}}}
\put(5266,-3706){\makebox(0,0)[lb]{\smash{{\SetFigFont{9}{10.8}{\familydefault}{\mddefault}{\updefault}$0$}}}}
\put(4186,-1546){\makebox(0,0)[lb]{\smash{{\SetFigFont{9}{10.8}{\familydefault}{\mddefault}{\updefault}$U'$}}}}
\put(1216,-5371){\makebox(0,0)[lb]{\smash{{\SetFigFont{9}{10.8}{\familydefault}{\mddefault}{\updefault}$U$}}}}
\put(8146,-5191){\makebox(0,0)[lb]{\smash{{\SetFigFont{9}{10.8}{\familydefault}{\mddefault}{\updefault}$U'$}}}}
\end{picture}%